\documentclass[11pt]{amsart}
\usepackage{amssymb,amsmath,amsfonts,amsthm}
\usepackage[margin=28mm]{geometry}
\usepackage{enumitem}
\setlist{nosep}
\usepackage[hidelinks]{hyperref}
\newtheorem{theorem}{Theorem}[section]
\newtheorem{lemma}[theorem]{Lemma}
\newtheorem{proposition}[theorem]{Proposition}
\newtheorem{corollary}[theorem]{Corollary}
\theoremstyle{definition}
\newtheorem{definition}[theorem]{Definition}
\newtheorem{example}[theorem]{Example}
\theoremstyle{remark}

\DeclareMathOperator{\dist}{dist}
\DeclareMathOperator{\tr}{tr}
\DeclareMathOperator{\diag}{diag}

\newcommand{\R}{\mathbb R}
\newcommand{\C}{\mathbb C}
\newcommand{\Slopes}{\mathcal S}
\DeclareMathOperator{\spec}{spec}
\DeclareMathOperator{\Hess}{Hess}
\newcommand{\Id}{\mathrm{Id}}
\newcommand{\dd}{\,d}
\numberwithin{equation}{section}
\title[Umbilic slopes and cubic Weingarten surfaces]{Umbilic slopes and cubic Weingarten surfaces}
\author{Haoxuan Cheng}
\address{School of Mathematical Sciences, Fudan University, Shanghai 200433, China}
\email{hxcheng25@m.fudan.edu.cn}
\date{}
\begin{document}
\begin{abstract}
We study umbilic slopes and the global classification of cubic
Weingarten surfaces. For a smooth surface in
Euclidean three-space, a nonconstant principal curvature germ with a unique
secant tangent has slope zero, minus one, infinity, an odd integer at least
three, or its reciprocal. At a nonconstant umbilic germ, a smooth regular
relation between mean and Gaussian curvature supplies such a tangent
after a continuous labeling of
the principal curvatures; the ordered curvature image has at most two
limiting directions, both in the same set. We also classify compact immersed surfaces without boundary of class
$C^4$ satisfying $\kappa_2=c\kappa_1^3$ in some order at every point,
for a fixed real $c$: each connected component is embedded as an
ellipsoid of revolution. The arguments share a classification of homogeneous Hessian
identities at an umbilic. Scalar unique continuation treats infinite-order
contact in the smooth slope problem, whereas an exact fourth-order growth
estimate and angular energy estimates yield the cubic classification.
\end{abstract}
\maketitle

\section{Introduction}\label{sec:introduction}
A relation between the principal curvatures constrains how a surface
can depart from its tangent sphere near an umbilic. We study this
constraint through a local $C^\infty$ restriction on the slope of the
principal curvature image and a global $C^4$ classification of compact
cubic Weingarten surfaces as ellipsoids of revolution.

These two questions appear in Yau's Problem~58
\cite{Yau1982}.
Hopf's analytic slope restriction is described in
\cite{Hopf1989}; Yau asks whether it
persists for smooth Weingarten surfaces of genus zero. The cubic
classification holds for analytic surfaces \cite{Simon2007}, whereas
Fern\'andez and Mira constructed compact $C^2$ counterexamples
\cite{FernandezMira2023}.

Complex-analytic and Codazzi approaches to Weingarten rigidity appear
in \cite{Bryant2011,AEG2010}. Uniqueness results for prescribed mean
curvature and for surface classes modeled by elliptic equations are
proved in \cite{GM2016,GM2020}. The quasiconformal curvature-diagram
condition in \cite{GMT2022} gives another extension of Hopf's sphere
theorem. The results below concern local slope restrictions and a
finite-regularity cubic relation.

For the definitions, let $X:M\to\R^3$ be a $C^2$ immersion of a
two-dimensional manifold. Higher regularity will be specified in each
statement. Let $p$ be an interior point of $M$. Choose a unit normal $N$ near $p$
and use the shape operator $A=-dN$, with the tangent space identified by
$dX$. Write
\[
 H=\tfrac12\tr A,\qquad K=\det A.
\]
Assume that $p$ is umbilic, so $A(p)=tI$ for some $t\in\R$.
A continuous labeling $(\kappa_1,\kappa_2)$ means two continuous real
functions near $p$ whose unordered values are the eigenvalues of $A$.
Such a labeling always exists by ordering the eigenvalues, but we do not
require that order. No choice of continuous principal direction fields is
part of this definition. Set
\[
 v(q)=(\kappa_1(q)-t,\kappa_2(q)-t).
\]
The curvature germ is \emph{nonconstant} if every neighborhood of $p$
contains a point where $v\ne0$. This property is independent of the labeling.

\begin{definition}
The curvature image for the chosen labeling has secant tangent $L$ at $p$
if $L\subset\R^2$ is a line through the origin and
\begin{equation}\label{eq:tangent}
 \lim_{\substack{q\to p\\v(q)\ne0}}
 \frac{\dist(v(q),L)}{|v(q)|}=0.
\end{equation}
The slope of $L$ is measured with $\kappa_1$ as the horizontal coordinate
and $\kappa_2$ as the vertical coordinate; the vertical line has slope
$\infty$.
\end{definition}
For a nonconstant germ, such a line is unique if it exists. The definition
uses an unoriented line, so approaching it from opposite directions makes
no difference. Introduce the set
\begin{equation}\label{eq:slopes}
 \Slopes=\{0,-1,\infty\}\cup
 \{2j+1,(2j+1)^{-1}:j=1,2,\ldots\}.
\end{equation}

\begin{theorem}\label{thm:slope}
Let $p$ be an interior umbilic of a $C^\infty$ immersion of a surface into
$\R^3$. Suppose its curvature germ is nonconstant. If a continuous labeling
of the principal curvatures has a secant tangent at $p$, its slope belongs
to $\Slopes$.
\end{theorem}

At a nonconstant umbilic germ satisfying a smooth regular relation
$\Phi(K,H)=0$, a suitable continuous labeling has such a tangent; every limiting direction for the ordered
curvatures also has slope in $\Slopes$. The precise statement and the
distinction between these two labelings appear in
Subsection~\ref{sec:relation}. Ordered tangent uniqueness is not asserted.

The global result concerns a fixed real constant $c$ and the unordered relation
\begin{equation}\label{eq:cubic}
 (\kappa_1-c\kappa_2^3)(\kappa_2-c\kappa_1^3)=0
 \quad\text{at every point of }M.
\end{equation}
This is precisely the statement that $\kappa_2=c\kappa_1^3$ holds in
some order at each point. The condition is independent of the eigenvalue
labels and of reversal of the local normal. We call an immersion satisfying it a cubic
immersion. A nonzero umbilic has $A=\lambda\Id$ with
$\lambda\ne0$.

\begin{theorem}\label{thm:main}
Let $X:M\to\R^3$ be a $C^4$ immersion of a nonempty compact smooth
two-dimensional manifold without boundary. Suppose that
\eqref{eq:cubic} holds for a fixed $c\in\R$. Then $c>0$, and the
restriction of $X$ to each connected component of $M$ is an embedding
onto an ellipsoid of revolution, including a round sphere. If $a_e>0$
and $b_e>0$ are its equatorial and axial semiaxes, respectively, then
\[
 b_e=\frac{a_e^2}{\sqrt c}.
\]
Conversely, every ellipsoid of revolution with this semiaxis relation
satisfies \eqref{eq:cubic} with the given constant $c$.
\end{theorem}

No orientability or embeddedness assumption is required. The ellipsoids
associated with different components may have different centers and axes.
Strict positivity of Gaussian curvature follows from the theorem.

Both proofs subtract the tangent sphere and use the homogeneous Hessian
classification in Section~\ref{sec:common}. The analytic difficulties
differ. In the smooth case, scalar unique continuation excludes
infinite-order contact and makes a first nonzero Taylor term available
(Section~\ref{sec:smooth-slopes}). For the cubic relation, the common
classification fixes the possible degree at four. A support-potential
growth estimate, using \cite{IKO2008,AM1992}, excludes a nonzero
$o(r^4)$ remainder; an angular energy estimate turns a nonzero radial
quartic term into actual rotational symmetry
(Section~\ref{sec:cubic-local}). A nonzero umbilic supplies the starting
point (Section~\ref{sec:existence}), and periodic uniqueness and
compactness complete the surface (Section~\ref{sec:cubic-global}).

Closed rotational examples and their curvature diagrams are studied in
\cite{KS2005}; direct integration of rotational models is developed in
\cite{CarreteroCastro2022}, and their umbilic regularity in
\cite{GuilfoyleRobson2025}. Here rotational symmetry is proved from
an arbitrary compact immersion. The cubic classification uses its own
contact and propagation arguments; it is not a consequence of the
smooth slope theorem. The general $C^3$ classification remains open
in this argument.

\section{The common local Hessian calculation}\label{sec:common}\label{sec:contact}
\subsection{Subtracting the tangent sphere}

Choose Euclidean coordinates centered at $X(p)$ so that the tangent
plane is horizontal and the chosen normal is upward. The immersion is
locally a graph $z=u(x,y)$, where
\[
 u(0)=0,\qquad Du(0)=0,\qquad D^2u(0)=tI.
\]
Throughout this section, $r=(x^2+y^2)^{1/2}$. Define the graph of the
same-oriented tangent sphere, or tangent plane when $t=0$, by
\[
 s_t(x,y)=\begin{cases}
 (1-\sqrt{1-t^2r^2})/t,&t\ne0,\\
 0,&t=0.
 \end{cases}
\]
For $t\ne0$ the disk is chosen so that $|t|r<1$. The difference
$h=u-s_t$ has zero Taylor coefficients through degree two. The graph identities below use only two derivatives. Later arguments
will specify either $C^\infty$ or $C^4$ regularity when they use Taylor
expansions or contact estimates.

For a vector $\xi\in\R^2$ and a symmetric real matrix $M$, define
\[
 \mathcal L(\xi)[M]
 =\frac{(I+\xi\otimes\xi)^{-1/2}M(I+\xi\otimes\xi)^{-1/2}}
 {\sqrt{1+|\xi|^2}}.
\]
The symmetric matrix $\mathcal A=\mathcal L(Du)[D^2u]$ represents the
shape operator in an orthonormal frame, and therefore has the principal
curvatures as eigenvalues. This follows from the graph metric
$I+Du\otimes Du$ and second fundamental form
$D^2u/\sqrt{1+|Du|^2}$. Put $B=\mathcal A-tI$.

\begin{lemma}\label{lem:error}
On a sufficiently small graph disk,
\begin{equation}\label{eq:shape-error}
 B=D^2h+E,\qquad
 |E|\le Cr^2|D^2h|+Cr|Dh|
\end{equation}
for a fixed constant $C$ and the Frobenius matrix norm.
\end{lemma}
\begin{proof}
The tangent sphere satisfies $\mathcal L(Ds_t)[D^2s_t]=tI$. Hence
\[
 B=\mathcal L(Du)[D^2h]
   +\bigl(\mathcal L(Du)-\mathcal L(Ds_t)\bigr)[D^2s_t].
\]
The map $\mathcal L$ is smooth and even in its vector argument, with
$\mathcal L(0)$ the identity on symmetric matrices. As $Du,Ds_t=O(r)$,
we have $\mathcal L(Du)-\mathcal L(0)=O(r^2)$. The derivative of
$\mathcal L$ is $O(r)$ along the segment joining $Du$ and $Ds_t$, so
\[
 \|\mathcal L(Du)-\mathcal L(Ds_t)\|\le Cr|Dh|.
\]
The Hessian of $s_t$ is bounded. The two estimates prove
\eqref{eq:shape-error}.
\end{proof}

\subsection{Classification of homogeneous Hessians}\label{sec:polynomial}

A homogeneous Taylor term inherits an algebraic relation between the
trace and determinant of its Hessian. The next classification treats
positive and negative slopes in the same calculation. The reality of
the polynomial is essential.

\begin{lemma}\label{lem:polynomial}
Let $P$ be a nonzero real homogeneous polynomial in two variables of
degree $n\ge3$, and let $\sigma\in\R\setminus\{0,-1\}$. If
\begin{equation}\label{eq:hessian}
 (\sigma+1)^2\det D^2P=\sigma(\tr D^2P)^2,
\end{equation}
then $n=2m$ for an integer $m\ge2$,
\[
 P(x,y)=a(x^2+y^2)^m,\qquad a\in\R\setminus\{0\},
 \qquad \sigma\in\{2m-1,(2m-1)^{-1}\}.
\]
Conversely, each listed polynomial satisfies \eqref{eq:hessian} for
both listed values of $\sigma$.
\end{lemma}
\begin{proof}
If $\sigma=1$, equation~\eqref{eq:hessian} says
$(P_{xx}-P_{yy})^2+4P_{xy}^2=0$. Hence $P_{xy}=0$ and
$P_{xx}=P_{yy}$. The first identity writes $P=f(x)+g(y)$; the second
makes $f''$ and $g''$ the same constant. This contradicts $n\ge3$.
We may therefore suppose $\sigma\ne1$.

Use independent complex variables $z,w$ and set
\[
 Q(z,w)=P\left(\frac{z+w}{2},\frac{z-w}{2i}\right),\qquad
 A_0=Q_{zw},\quad B_0=Q_{zz},\quad C_0=Q_{ww}.
\]
Reality means
$Q(z,w)=\overline{Q(\bar w,\bar z)}$.
On the real slice $w=\bar z$ the trace and determinant in
\eqref{eq:hessian} are $4A_0$ and $4(A_0^2-B_0C_0)$.
The invertible complex linear change of variables therefore turns
\eqref{eq:hessian} into a polynomial identity
\begin{equation}\label{eq:complex-hessian}
 B_0C_0=\lambda A_0^2,\qquad
 \lambda=\frac{(\sigma-1)^2}{(\sigma+1)^2}\notin\{0,1\}.
\end{equation}
If $A_0=0$, homogeneity and reality give
$Q=bz^n+\bar b w^n$ with $b\ne0$; then
$B_0C_0=n^2(n-1)^2|b|^2z^{n-2}w^{n-2}\ne0$, a contradiction.
Thus all three polynomials $A_0,B_0,C_0$ are nonzero.

We claim that $A_0$ has no projective root away from the coordinate
axes. Suppose $A_0(\tau,1)=0$ for $\tau\ne0$. Equation
\eqref{eq:complex-hessian} makes $B_0$ or $C_0$ vanish there.
By exchanging $z,w$ and replacing $\tau$ by $1/\tau$ if necessary,
we may assume $B_0(\tau,1)=0$. Put $p(\zeta)=Q(\zeta,1)$ and
$g=p'$. Homogeneity gives
\begin{align}
 B_0(\zeta,1)&=g'(\zeta),\notag\\
 A_0(\zeta,1)&=(n-1)g(\zeta)-\zeta g'(\zeta),\label{eq:root-formulas}\\
 C_0(\zeta,1)&=n(n-1)p(\zeta)-2(n-1)\zeta p'(\zeta)
                      +\zeta^2p''(\zeta).\notag
\end{align}
Thus $g(\tau)=g'(\tau)=0$. The nonzero polynomial $g$ has an
expansion $g(\zeta)=d(\zeta-\tau)^\ell+\cdots$ with $d\ne0$ and
$\ell\ge2$. Both $A_0(\zeta,1)$ and $B_0(\zeta,1)$ have vanishing
order $\ell-1$ at $\tau$, and their leading coefficients have ratio
$-\tau$. Taking orders in \eqref{eq:complex-hessian} forces
$C_0(\zeta,1)$ to have order $\ell-1$ as well. In particular,
$C_0(\tau,1)=0$; equation~\eqref{eq:root-formulas} now gives
$p(\tau)=0$.

It follows that $p$ has order $\ell+1$ at $\tau$. In the last line of
\eqref{eq:root-formulas}, the sole term of order $\ell-1$ is
$\zeta^2p''(\zeta)$. Comparing leading coefficients in
\eqref{eq:complex-hessian} gives
\[
 \tau^2(\ell d)^2=\lambda\tau^2(\ell d)^2.
\]
Since $\tau\ell d\ne0$, this contradicts $\lambda\ne1$ and proves
the claim.

Factorization of a complex polynomial now shows that
$A_0=K_0z^kw^{n-2-k}$ for some $K_0\ne0$. Its reality forces
$k=n-2-k$ and $K_0\in\R$. Hence $n=2m$, $m\ge2$, and
$A_0=K_0z^{m-1}w^{m-1}$. Integrating the mixed derivative, with
homogeneity and reality, yields
\[
 Q=az^mw^m+bz^{2m}+\bar b w^{2m},
 \qquad a=K_0/m^2\in\R\setminus\{0\}.
\]
The coefficient of $z^{3m-2}w^{m-2}$ in
\eqref{eq:complex-hessian} is
$2m(2m-1)b\,a m(m-1)$ on the left and zero on the right.
This monomial differs from the central monomial and the other cross
term, since $m\ge2$. Therefore $b=0$ and $P=ar^{2m}$.
Its radial and tangential Hessian eigenvalues are
\[
 2ma(2m-1)r^{2m-2},\qquad 2ma r^{2m-2}.
\]
Substituting them in \eqref{eq:hessian} gives precisely
$\sigma=2m-1$ or $\sigma=(2m-1)^{-1}$, and also proves the converse.
\end{proof}

\section{Smooth umbilic slopes}\label{sec:smooth-slopes}
\subsection{Finite contact}\label{sec:beltrami}

The planar representation method rests on the measurable Beltrami
equation; see \cite{AB1960,Bojarski2010}. Related reductions occur in
Alessandrini's strong unique continuation theorem for divergence-form
equations \cite{Alessandrini2012}. Positive multipliers and first-order
complex equations also enter quantitative results of Davey
\cite{Davey2020} and Le Balc'h--Souza \cite{LeBalchSouza2024}.
Here we need only the following local scalar statement.

We first record the analytic statement used to exclude infinite contact
with a tangent sphere. For $z=x+iy$, set
\[
 \partial_z=\tfrac12(\partial_x-i\partial_y),\qquad
 \partial_{\bar z}=\tfrac12(\partial_x+i\partial_y).
\]
A scalar function is said to vanish to infinite order at the origin if,
for every integer $N\ge1$, its absolute value is at most $C_N|z|^N$
in some neighborhood of the origin. We use this pointwise formulation.

\begin{lemma}\label{lem:sucp}
Let $D\subset\C$ be a disk containing $0$, and let $F:D\to\C$ be $C^1$.
Suppose $b:D\to[0,\infty)$ is locally bounded and measurable, and
$0\le k<1$ is a constant such that
\begin{equation}\label{eq:beltrami-inequality}
 |F_{\bar z}|\le k|F_z|+b|F|
\end{equation}
almost everywhere. If $F$ vanishes to infinite order at $0$, then $F$
vanishes identically on a neighborhood of $0$.
\end{lemma}

\begin{proof}
Choose a smaller disk $D_0$ centered at $0$ with closure in $D$, on which
$b\le B$ for a finite constant $B$. We first express the inequality as
an equation
\begin{equation}\label{eq:inhom-beltrami}
 F_{\bar z}=\mu F_z+aF,\qquad |\mu|\le k,\quad |a|\le B,
\end{equation}
with measurable coefficients. Define them on the full-measure set where
\eqref{eq:beltrami-inequality} holds, and set both coefficients to zero
on its null complement. At a point where
$|F_{\bar z}|\le k|F_z|$, put $a=0$ and $\mu=F_{\bar z}/F_z$ if
$F_z\ne0$; put $\mu=0$ if $F_z=0$.
On the other set, write $X=F_z$ and $Y=F_{\bar z}$. There $Y\ne0$,
and define
\[
 \mu=\begin{cases}
 k|X|Y/(|Y|X),&X\ne0,\\
 0,&X=0,
 \end{cases}
 \qquad R=Y-\mu X.
\]
Then $|R|=|Y|-k|X|\le b|F|$. On this set $F\ne0$ by
\eqref{eq:beltrami-inequality}, so $a=R/F$ is defined and bounded by $B$.
This proves \eqref{eq:inhom-beltrami}, including the zero sets involved
in the construction.

Extend $\mu$ and $a$ by zero outside $D_0$. For compactly supported
functions $\omega\in L^p(\C)$ with $p>2$, the Cauchy transform and its
$z$ derivative
are denoted by
\[
 T\omega(z)=\frac1\pi\int_{\C}\frac{\omega(\zeta)}{z-\zeta}\,dA(\zeta),
 \qquad S\omega=\partial_zT\omega.
\]
Here $dA$ is planar area measure, and derivatives of these transforms
are taken in the distributional sense. Thus
$\partial_{\bar z}T\omega=\omega$, and $S$ is the Beurling transform.
It is bounded on $L^p(\C)$ for $1<p<\infty$, its $L^2$ norm is one,
and its operator norm tends to one as $p\to2$; see
\cite[pp.~60--61]{Bojarski2010}.
Choose $p>2$ sufficiently close to $2$ that
$k\|S\|_{L^p\to L^p}<1$; if $k=0$, any finite $p>2$ is suitable.
The Neumann series defines
\begin{equation}\label{eq:similarity}
 \omega=(I-\mu S)^{-1}a\in L^p(\C),\qquad s=T\omega.
\end{equation}
The equation $\omega-\mu S\omega=a$ shows that $\omega$ is supported
in $\overline{D_0}$. Consequently $s\in W^{1,p}_{\mathrm{loc}}(\C)$,
and $s$ is locally bounded. For the latter assertion, apply H\"older's
inequality to the Cauchy integral: its kernel has locally integrable
$p'$-th power, where $p'=p/(p-1)<2$, and the support of $\omega$ is fixed
and compact. Sobolev embedding gives a continuous representative of $s$.

Put $G=e^{-s}F$ on $D_0$. The Sobolev product and chain rules apply on
smaller disks because $s$ is locally bounded and $F$ is $C^1$. Equations
\eqref{eq:inhom-beltrami} and \eqref{eq:similarity} give
\begin{equation}\label{eq:hom-beltrami}
 G_{\bar z}-\mu G_z
 =e^{-s}\bigl(F_{\bar z}-\mu F_z-F(s_{\bar z}-\mu s_z)\bigr)=0.
\end{equation}
This is the local similarity construction used in
\cite[Lemma~4]{Bojarski2010}. It only requires the exponent $p>2$
chosen above, not invertibility on an arbitrary prescribed $L^p$ space.

The measurable Riemann mapping theorem for compactly supported $\mu$
with $\|\mu\|_\infty<1$ gives a quasiconformal homeomorphism
$\chi:\C\to\C$ solving $\chi_{\bar z}=\mu\chi_z$. We choose the
principal solution, normalized by $\chi(z)=z+O(1/z)$ at infinity; see
\cite[Proposition~3]{Bojarski2010}. The scalar factorization of a
$W^{1,2}_{\mathrm{loc}}$ solution of \eqref{eq:hom-beltrami} gives
\[
 G=\mathcal H\circ\chi
\]
with $\mathcal H$ holomorphic on $\chi(D_0)$; see
\cite[p.~26, equation~(5)]{Lorent2016}.
This also covers the constant case, or that case can be separated first.
The inverse principal solution belongs locally to $W^{1,q}$ for some
$q>2$; see \cite[p.~66]{Bojarski2010}. Applying the two-dimensional
Sobolev--Morrey embedding on a disk compactly contained in $\chi(D_0)$
gives, near $w_0=\chi(0)$,
\begin{equation}\label{eq:inverse-holder}
 |\chi^{-1}(w)-\chi^{-1}(w_0)|\le C|w-w_0|^\gamma,
 \qquad \gamma=1-2/q>0.
\end{equation}
When $\mu=0$ one can take $\chi$ to be the identity.

Since $e^{-s}$ is bounded near $0$, $G$ also vanishes to infinite order.
If $\mathcal H$ is not identically zero near $w_0$, its zero at $w_0$
has finite order $m\ge1$. Hence
$|\mathcal H(w)|\ge c|w-w_0|^m$ near $w_0$.
Combining this lower bound with \eqref{eq:inverse-holder} yields
\[
 |G(z)|\ge c'|z|^{m/\gamma}
\]
near $0$. Choosing an integer $N>m/\gamma$ contradicts infinite-order
vanishing. Thus $\mathcal H$, and then $F$, vanishes near the origin.
\end{proof}

We apply the scalar lemma to the difference $h=u-s_t$ introduced in
Section~\ref{sec:common}. In this section the immersion is $C^\infty$.
A finite nonzero secant slope will put a complex gradient of $h$ in
the strict distortion range of Lemma~\ref{lem:sucp}.

\begin{proposition}\label{prop:finite-contact}
Under the hypotheses of Theorem~\ref{thm:slope}, suppose the secant
slope $\sigma$ is finite and nonzero. Then $h=u-s_t$ has a first nonzero
homogeneous Taylor term $P_n$ of integer degree $n\ge3$.
\end{proposition}
\begin{proof}
For a symmetric matrix $M$, write
\[
 a(M)=\tfrac12\tr M,\qquad
 b(M)=\tfrac12(M_{11}-M_{22}+2iM_{12}).
\]
These quantities describe its trace and traceless parts. They satisfy
\[
 |M|^2=2\bigl(|a(M)|^2+|b(M)|^2\bigr).
\]
The eigenvalues of $B$ are $\alpha=\kappa_1-t$ and
$\beta=\kappa_2-t$. In particular
\[
 a(B)=\tfrac12(\alpha+\beta),\qquad
 |b(B)|=\tfrac12|\alpha-\beta|.
\]
If $\sigma>0$, condition~\eqref{eq:tangent} implies
\[
 \frac{|\alpha-\beta|}{|\alpha+\beta|}
 \longrightarrow\frac{|1-\sigma|}{1+\sigma}<1
\]
at nonzero curvature increments. The denominator is nonzero near $p$
on this set, because it is nonzero on the limiting line away from the
origin. Thus $|b(B)|\le k_0|a(B)|$ for some fixed $k_0<1$ after
shrinking the disk. At $B=0$ the same inequality holds.
If $\sigma<0$, the corresponding estimate is
\[
 \frac{|\alpha+\beta|}{|\alpha-\beta|}
 \longrightarrow\frac{|1+\sigma|}{1-\sigma}<1,
 \qquad |a(B)|\le k_0|b(B)|.
\]

Substitute \eqref{eq:shape-error}. In the positive case, norm equivalence
gives, on a disk of radius tending to zero,
\[
 (1-\varepsilon)|b(D^2h)|
 \le(k_0+\varepsilon)|a(D^2h)|+Cr|Dh|,
 \qquad \varepsilon\longrightarrow0.
\]
Choose the radius so that $\varepsilon<(1-k_0)/2$. For
$F=h_x+ih_y$ we have $F_z=a(D^2h)$ and
$F_{\bar z}=b(D^2h)$. Dividing by $1-\varepsilon$ gives
\begin{equation}\label{eq:gradient-beltrami}
 |F_{\bar z}|\le k|F_z|+Cr|F|,\qquad k<1.
\end{equation}
In the negative case the same absorption interchanges $a$ and $b$.
Taking $F=h_x-ih_y$ then gives
$F_{\bar z}=a(D^2h)$ and $F_z=\overline{b(D^2h)}$, so
\eqref{eq:gradient-beltrami} again holds.

If $h$ had no nonzero Taylor term, smoothness would make $h$ and its
first derivatives vanish to infinite order. Lemma~\ref{lem:sucp}, with
the locally bounded coefficient $Cr$, would force $Dh=0$ near $0$.
Since $h(0)=0$, this would give $u=s_t$ there, contradicting the
nonconstant curvature germ. Thus a first term exists, and its degree
is at least three because the Taylor coefficients through degree two
already vanish.
\end{proof}

\subsection{The secant slope theorem}\label{sec:main-proof}

The geometric estimate and polynomial classification now apply to the
same first Taylor term. A symmetric identity allows us to pass to that
term without differentiating the principal curvature labels.

\begin{proof}[Proof of Theorem~\ref{thm:slope}]
Write $\sigma$ for the secant slope. The values $0,\infty,-1$ are already
in $\Slopes$, so assume $\sigma\in\R\setminus\{0,-1\}$.
Proposition~\ref{prop:finite-contact} gives a first nonzero homogeneous
term $P_n$ of $h=u-s_t$, with $n\ge3$. Taylor's theorem with derivative
control and Lemma~\ref{lem:error} give
\begin{equation}\label{eq:leading-shape}
 B=D^2P_n+O(r^{n-1}).
\end{equation}
Indeed, $Dh=O(r^{n-1})$, $D^2h=D^2P_n+O(r^{n-1})$, and the two
error terms in \eqref{eq:shape-error} are $O(r^n)$.

For $\alpha=\kappa_1-t$ and $\beta=\kappa_2-t$, the finite slope
condition is
$\beta-\sigma\alpha=o((\alpha^2+\beta^2)^{1/2})$ at nonzero increments.
The exact symmetric identity
\begin{equation}\label{eq:symmetric-residual}
 (\sigma+1)^2\det B-\sigma(\tr B)^2
   =(\beta-\sigma\alpha)(\alpha-\sigma\beta)
\end{equation}
therefore has right-hand side $o(|B|^2)$. At $B=0$ both sides vanish,
so this estimate holds throughout the shrinking neighborhood.
For any unit vector $e\in\R^2$, evaluate at $(x,y)=re$ and divide
by $r^{2n-4}$. Equation~\eqref{eq:leading-shape} and
$|B|=O(r^{n-2})$ yield
\[
 (\sigma+1)^2\det D^2P_n(e)=\sigma(\tr D^2P_n(e))^2.
\]
This argument remains valid when $D^2P_n(e)=0$; no division by that
matrix or by either eigenvalue has been made. Homogeneity extends
the identity to the plane. Lemma~\ref{lem:polynomial} gives
$\sigma=2m-1$ or its reciprocal for some integer $m\ge2$, proving the result.
\end{proof}

For comparison with formulations using only nonumbilic points, we record
the following equivalence. It also explains why open pieces of a sphere
cannot hide a different limiting direction.

\begin{lemma}\label{lem:nonumbilic}
Let $X$ be a $C^3$ surface immersion in $\R^3$, let $p$ be an interior
umbilic with nonconstant curvature germ, and choose continuous principal
curvature labels near $p$. For a fixed line $L$, condition
\eqref{eq:tangent} holds if and only if it holds when $q$ is restricted
to nonumbilic points.
\end{lemma}
\begin{proof}
Only the reverse implication requires proof. Work in a coordinate disk
centered at $p$. Let $Z$ be its closed set of umbilics and $U$ the
interior of $Z$. On $U$ the shape operator is $HI$; the Euclidean
Codazzi identity gives
$(VH)W=(WH)V$ for tangent fields $V,W$. Taking two independent
orthonormal directions shows $dH=0$. Thus $H$ is constant on each
connected component of $U$.

Consider an umbilic sequence $q_j\to p$ with $v(q_j)\ne0$.
If $q_j\notin U$, set $r_j=q_j$. Otherwise let $U_j$ be the component
of $U$ containing $q_j$, and let its mean curvature be $h_j$.
Then $h_j\ne t$. Follow the coordinate segment from $q_j$ to $p$
until its first exit from $U_j$. This exit occurs before $p$, since
continuity at $p$ would otherwise imply $h_j=t$. Call the exit point
$r_j$. It satisfies $|r_j-p|\le|q_j-p|$ and
$v(r_j)=v(q_j)=(h_j-t,h_j-t)$.
Every such $r_j$ is in the closure of the nonumbilic set. Indeed, a
connected umbilic neighborhood of $r_j$ would lie in $U$ and meet
$U_j$, putting $r_j$ in that same open component, a contradiction.
The same closure assertion holds when $r_j=q_j\notin U$.

By continuity choose a nonumbilic $x_j$ so close to $r_j$ that
\[
 |x_j-r_j|<1/j,\qquad
 |v(x_j)-v(r_j)|<|v(r_j)|/j.
\]
Then $x_j\to p$ and, for $j>1$, its increment is nonzero. The
normalized directions of $v(x_j)$ and $v(q_j)$ differ by at most
$2/j$. Thus convergence to $L$ on the nonumbilic set implies the
same convergence on the chosen umbilic sequence. Combining the two
types of sequence proves \eqref{eq:tangent}.
Finally, nonconstant curvature implies that nonumbilic points
accumulate at $p$: otherwise the Codazzi argument on a whole disk
would force $A=tI$ there. The condition is therefore not vacuous.
\end{proof}

\begin{corollary}\label{cor:curvature-relation}
In the setting of Theorem~\ref{thm:slope}, suppose a continuous labeling
satisfies $W(\kappa_1,\kappa_2)=0$ near $p$, where $W$ is $C^1$
near $(t,t)$ and $dW(t,t)\ne0$. Then its secant tangent exists and
has slope in $\Slopes$.
\end{corollary}
\begin{proof}
The first-order expansion of $W$ gives
\[
 dW(t,t)\cdot v(q)=o(|v(q)|).
\]
The kernel of the nonzero linear functional $dW(t,t)$ is a line,
and this equation is exactly \eqref{eq:tangent} for that line.
Apply Theorem~\ref{thm:slope}.
\end{proof}

\subsection{Regular relations between mean and Gaussian curvature}\label{sec:relation}

A regular relation in $(K,H)$ need not give a regular relation in the
individual principal curvatures at an umbilic. The coordinate change
has a critical point there. We prove a normal form that retains this
possibility and controls the necessary relabeling.

For a nonzero vector $w\in\R^2$, denote its line by $[w]\in\R\mathbb P^1$.
For the ordered curvatures $\kappa_+\ge\kappa_-$, a limiting secant direction
means any limit in $\R\mathbb P^1$ of
$[(\kappa_+(q)-t,\kappa_-(q)-t)]$ along a sequence $q\to p$ with nonzero
increment. Compactness of $\R\mathbb P^1$ guarantees at least one such
limit for a nonconstant germ.

\begin{corollary}\label{thm:relation}
Let $p$ be an interior umbilic of a $C^\infty$ immersed surface in $\R^3$,
with $A(p)=tI$ and nonconstant curvature germ. Suppose that
$\Phi$ is a real $C^\infty$ function on a neighborhood of $(t^2,t)$ in
$\R^2$, that $\Phi(K,H)=0$ near $p$, and that
$d\Phi(t^2,t)\ne0$. Then there is a continuous labeling of the principal
curvatures with a secant tangent at $p$, and its slope belongs to
$\Slopes$. For the ordered curvatures, every limiting secant direction
has slope in $\Slopes$. There are at most two such directions; if there
are two, they are exchanged by interchanging the curvature coordinates.
\end{corollary}

This applies in particular at every nonconstant umbilic germ of a closed
genus-zero surface satisfying the hypotheses. It gives a tangent for
suitable continuous labels, while the ordered image may have two
limiting directions. The local examples below do not decide whether
closedness imposes ordered tangent uniqueness.

\begin{proof}[Proof of Corollary~\ref{thm:relation}]
Start with the ordered curvatures and set
\[
 x=H-t,\qquad y=\tfrac12(\kappa_+-\kappa_-)\ge0.
\]
These symbols are scalar curvature quantities in this subsection, not
coordinates on the surface. They satisfy
\[
 \kappa_\pm-t=x\pm y,\qquad K=t^2+2tx+x^2-y^2.
\]
Introduce an independent real variable $z$ and the smooth function
\[
 \Psi(x,z)=\Phi(t^2+2tx+x^2-z,t+x).
\]
All derivatives of $\Phi$ in this proof are evaluated at $(t^2,t)$.
Write
\[
 A_1=2t\Phi_K+\Phi_H,\qquad C_1=\Phi_K.
\]
Then $\Psi_x(0,0)=A_1$ and $\Psi_z(0,0)=-C_1$. Regularity of
$\Phi$ implies that $A_1,C_1$ are not both zero, and the actual
curvatures satisfy $\Psi(x,y^2)=0$.

If $A_1\ne0$, the implicit function theorem writes this relation
as $x=g(y^2)$, where $g$ is smooth and $g(0)=0$.
Consequently $x=O(y^2)$. A nonzero curvature increment has $y\ne0$,
so $x/y\to0$. Its ordered direction therefore tends to $[1:-1]$,
and the slope is $-1$. The ordered continuous labeling already
proves all assertions in this case.

Suppose now that $A_1=0$. Then $C_1\ne0$, and the implicit function
theorem instead gives $y^2=f(x)$ with $f(0)=f'(0)=0$. Taylor's
integral formula defines a smooth function
\[
 h(x)=\int_0^1(1-s)f''(sx)\,ds,
 \qquad y^2=x^2h(x).
\]
Let $\eta=h(0)$. If $\eta<0$, the last identity forces $x=y=0$
near $p$, contrary to the nonconstant germ. If $\eta=0$, then at
nonzero increments $x\ne0$ and $|y/x|=\sqrt{h(x)}\to0$.
The ordered curvatures have secant slope $1$, contradicting
Theorem~\ref{thm:slope}. Thus $\eta>0$.

Shrink the neighborhood so that $h>0$. Define two functions on the
surface by
\begin{equation}\label{eq:relabel}
 \lambda_1=t+x+x\sqrt{h(x)},\qquad
 \lambda_2=t+x-x\sqrt{h(x)}.
\end{equation}
They are smooth functions of $H$. Their sum is $2H$ and their
product is $K$, so they are precisely the two principal curvatures,
with a possibly different labeling. At every nonzero increment $x\ne0$,
and their secant directions tend to
\[
 [1+\sqrt\eta:1-\sqrt\eta].
\]
In particular this labeling has a unique finite slope
\begin{equation}\label{eq:normal-slope}
 s=\frac{1-\sqrt\eta}{1+\sqrt\eta}.
\end{equation}
By Theorem~\ref{thm:slope}, $s\in\Slopes$. If $\eta>1$, then
$-1<s<0$, which is impossible. If $\eta=1$, then $s=0$.
If $0<\eta<1$, then $0<s<1$, so
\begin{equation}\label{eq:eta}
 s=(2j+1)^{-1},\qquad
 \eta=\frac{j^2}{(j+1)^2}
 \quad\text{for some integer }j\ge1.
\end{equation}

For the ordered labels, $y=|x|\sqrt{h(x)}$. On the set $x>0$ their
secant direction tends to $[1+\sqrt\eta:1-\sqrt\eta]$, and on
$x<0$ it tends to $[1-\sqrt\eta:1+\sqrt\eta]$. Every limiting
direction is one of these two: any sequence admits a subsequence
with a fixed sign of $x$. The two lines are exchanged by the
coordinates. Their slopes are $0,\infty$ when $\eta=1$, or the
reciprocal pair in \eqref{eq:eta} when $0<\eta<1$.
Only a side accumulating at $p$ contributes a direction.
This proves the ordered assertion as well as the existence of the
labeling \eqref{eq:relabel}.
\end{proof}

The proof identifies a restriction on the second derivative of the
relation in the degenerate case $A_1=0$. Evaluate all derivatives below
at $(t^2,t)$ and recall $C_1=\Phi_K(t^2,t)\ne0$. With
\[
 D_1=4t^2\Phi_{KK}+4t\Phi_{KH}+\Phi_{HH},
\]
implicit differentiation gives
\[
 \eta=\frac{f''(0)}2=1+\frac{D_1}{2C_1}.
\]
Thus, at a nonconstant umbilic germ with $A_1=0$, this number must
belong to $\{1\}\cup\{j^2/(j+1)^2:j\ge1\}$. This is a necessary
condition on a relation realized by a smooth surface, not an existence
assertion for every relation with such a second derivative.

\subsubsection*{Ordered tangents and local examples}\label{sec:examples}

The normal form distinguishes a unique tangent for suitable labels from
a unique tangent for ordered labels. The following local graphs show
that this distinction occurs for actual immersed surfaces.

\begin{example}\label{ex:cylinder}
Consider $X(u,v)=(u,v,u^3)$ with upward normal. Its first and second
fundamental forms have matrices
\[
 \diag(1+9u^4,1),\qquad
 \diag\bigl(6u/\sqrt{1+9u^4},0\bigr).
\]
The principal curvatures along the parameter directions are
\[
 a(u)=\frac{6u}{(1+9u^4)^{3/2}},\qquad b(u)=0.
\]
The line $u=0$ consists of umbilics with $t=0$. The relation
$\Phi(K,H)=K=0$ is smooth and regular. For $u>0$ the ordered pair
is $(a(u),0)$; for $u<0$ it is $(0,a(u))$. At $(0,0)$ the ordered
curvature image therefore has both horizontal and vertical limiting
secants. With the continuous labels $(a,0)$ it has just the horizontal
secant tangent.
\end{example}

\begin{example}\label{ex:nonzero}
A similar phenomenon occurs at nonzero umbilic curvature. For $r$ near
$1$ let $f(r)=(r-1)+(r-1)^2$ and consider
\[
 X(r,\theta)=(r\cos\theta,r\sin\theta,f(r)).
\]
With upward normal, the meridian and parallel curvatures are
\[
 a(r)=\frac{f''(r)}{(1+f'(r)^2)^{3/2}},\qquad
 b(r)=\frac{f'(r)}{r\sqrt{1+f'(r)^2}}.
\]
At $r=1$ direct differentiation gives
\[
 a(1)=b(1)=t=1/\sqrt2,\qquad
 a'(1)=-3/\sqrt2,\qquad b'(1)=0.
\]
Thus $H'(1)=-3/(2\sqrt2)\ne0$. Let $r=R(H)$ be its local smooth
inverse and put $G(H)=a(R(H))b(R(H))$. The surface satisfies the
regular smooth relation $\Phi(K,H)=K-G(H)=0$, with $\Phi_K=1$.
For $\delta=r-1$,
\[
 a(r)-t=-3\delta/\sqrt2+O(\delta^2),\qquad
 b(r)-t=O(\delta^2).
\]
The ordering switches across $r=1$. The ordered secant directions
are horizontal on one side and vertical on the other, whereas the
labels $(a,b)$ give a single horizontal tangent.
\end{example}

Both examples are local and have nonisolated umbilics. They show that
regularity of $\Phi$ and a nonconstant germ do not by themselves imply
ordered tangent existence. They do not settle whether a global regular
relation on a closed genus-zero surface imposes a further restriction.

The positive values in \eqref{eq:slopes} already occur among local
smooth graphs. For $a\ne0$ and an integer $m\ge2$, take
$u(r)=ar^{2m}$. The meridian and parallel curvatures at the origin
both vanish, while for $r>0$ they are
\[
 \kappa_{\mathrm{mer}}=
 \frac{2m(2m-1)a r^{2m-2}}
 {(1+4m^2a^2r^{4m-2})^{3/2}},\qquad
 \kappa_{\mathrm{par}}=
 \frac{2ma r^{2m-2}}
 {\sqrt{1+4m^2a^2r^{4m-2}}}.
\]
The labeling $(\kappa_{\mathrm{mer}},\kappa_{\mathrm{par}})$ has
secant slope $(2m-1)^{-1}$; exchanging the labels gives $2m-1$.
This verifies local realization of the positive slopes in
Theorem~\ref{thm:slope}. It does not impose a prescribed regular
relation or provide closed realizations.

Finally, infinite differentiability enters when the first nonzero
Taylor term is selected after infinite-order contact has been excluded.
No fixed finite degree is available in Theorem~\ref{thm:slope}.
The argument therefore does not establish the same statement for a
fixed finite differentiability class. Nor does a radial leading term
alone establish rotational symmetry of the surface. For the cubic relation, the coefficient in the homogeneous identity
fixes the possible degree at four. The remaining sections use the exact
equation to control contact beyond that degree, turn the radial term
into a rotational neighborhood, and complete the surface globally.

\section{A nonzero umbilic in every positive-curvature component}\label{sec:existence}
We now turn to the cubic relation \eqref{eq:cubic}. Its global
classification needs a nonzero umbilic as a starting point.
We find one in each positive-curvature component, without assuming
that this component is the whole surface or has a regular boundary.
Throughout this section it suffices that
$X:M\to\R^3$ is a $C^3$ immersion of a nonempty compact smooth
surface without boundary and satisfies \eqref{eq:cubic}.

\begin{lemma}\label{lem:positive}
Under the preceding $C^3$ hypotheses, one has $c>0$ and $K\geq0$.
Each connected component of $M$ contains a point with $K>0$.
\end{lemma}
\begin{proof}
Fix $a\in\R^3$ and maximize $|X-a|^2/2$ on a connected component. Its
maximum is $R^2/2>0$, since an immersion is not constant. At a maximum $p$,
choose the local normal $N=(a-X(p))/R$. The induced metric $g$ and second
fundamental form $h(V,W)=g(AV,W)$ satisfy
\[
 0\geq\Hess_M(|X-a|^2/2)(V,V)=g(V,V)-Rh(V,V).
\]
Thus $A_p\geq R^{-1}\Id$. Both principal curvatures are positive there,
so the cubic relation forces $c>0$. At any point the same relation gives
$K=c\kappa_i^4\geq0$, for the curvature chosen as its cubic input.
\end{proof}

We replace $X$ by $X/\sqrt c$ until the final rescaling. Its principal
curvatures are multiplied by $\sqrt c$, so the normalized relation has
constant $1$. Write $\Omega=\{K>0\}$. On $\Omega$ there is a unique local
normal for which $A$ is positive definite. These normals agree on overlaps
and orient each component of $\Omega$. There is a positive $C^1$ function
$t$ on $\Omega$ such that
\begin{equation}\label{eq:t-spectrum}
 \spec A=\{t,t^3\},\qquad K=t^4,\qquad \tr A=t+t^3.
\end{equation}
Indeed, $s\mapsto s+s^3$ has positive derivative. In particular, $t=1$
is precisely the nonzero umbilic condition.

The mean curvature vector $\mathbf H=(\tr A)N/2$ is a globally defined
 $C^1$ $\R^3$-valued function, even if $M$ is not orientable. Define
$\tau:M\to[0,\infty)$ by $\tau+\tau^3=2|\mathbf H|$. This function is
continuous, agrees with $t$ on $\Omega$, and vanishes exactly off $\Omega$:
the cubic relation makes $K=0$ equivalent to $A=0$. Compactness gives
\begin{equation}\label{eq:gradient-bound}
 |\nabla t|\leq L:=2\sup_M\|D\mathbf H\|<\infty
 \quad\hbox{on }\Omega.
\end{equation}
Here $D\mathbf H$ is the differential into $\R^3$, with operator norm
for the induced metric. To prove the bound, differentiate $t+t^3=2|\mathbf H|$
at points where $\mathbf H\ne0$ and use $1+3t^2\geq1$.
Only continuity of $\tau$ is used at its zero set.

We record the local structure equations for later use. On a nonumbilic
part of $\Omega$, choose a positively oriented $C^1$ orthonormal principal frame
$e_1,e_2$ with eigenvalues $t,t^3$ and dual forms $\theta_1,\theta_2$.
Set $D=1-t^2$, $p=e_1t$, $q=e_2t$, and define the connection form by
$\nabla e_1=\omega e_2$, where $\omega=\alpha\theta_1+\beta\theta_2$.
The two components of the Codazzi equation give
\begin{equation}\label{eq:codazzi}
 \alpha=\frac{q}{tD},\qquad \beta=\frac{3tp}{D},\qquad
 dt\wedge\omega=\frac{3tp^2-q^2/t}{D}\dd\mu.
\end{equation}
Here $d\mu=\theta_1\wedge\theta_2$ is the oriented area form.
The two sides of Codazzi have components
$q e_1+\beta(t-t^3)e_2$ and
$\alpha(t-t^3)e_1+3t^2p e_2$. The induced metric is $C^2$, whereas the principal connection form
$\omega$ is in general only continuous. The following weak form of the
Gauss equation is enough for both the cutoff argument and the later
principal-coordinate calculation.

\begin{lemma}\label{lem:weak-gauss}
On the positive nonumbilic set, the connection form satisfies
\begin{equation}\label{eq:connection-curvature}
 d\omega=-K\dd\mu
\end{equation}
in the sense of distributions. If $f$ is $C^1$, then
$d(f\omega)=df\wedge\omega-fK\dd\mu$. A compactly supported continuous
one-form whose distributional exterior derivative is continuous has
integral of that exterior derivative equal to zero.
\end{lemma}
\begin{proof}
In a coordinate disk, apply Gram--Schmidt to a smooth coordinate frame
using the $C^2$ metric. This gives a $C^2$ orthonormal frame with a $C^1$
connection form $\omega_0$. A local $C^1$ angle $\gamma$ rotates this
frame to the principal frame. Thus $\omega=\omega_0+d\gamma$ and
$d\omega=d\omega_0=-K\dd\mu$ distributionally. The last equality is
the classical Gauss equation in the reference frame; for a $C^3$
immersion its intrinsic curvature equals $\det A$.

The product rule follows from the same decomposition. In particular,
$d(f\,d\gamma)=df\wedge d\gamma$, as can also be checked by local
$C^1$ approximation of $f$ and $\gamma$ by smooth functions.
For the last assertion, use a smooth partition of unity on the compact
support of the form. Each resulting coordinate form has compact support;
its distributional derivatives pair to zero with a test function equal
to one near that support. Summing gives the assertion, since the
derivatives of the partition functions cancel.
\end{proof}

\begin{proposition}\label{prop:umbilic}
Every connected component $U$ of $\Omega$ contains a point where $t=1$.
\end{proposition}
\begin{proof}
Suppose that $U$ contains no such point. Then the eigenvalues are distinct
throughout $U$. Two oriented principal frames with the prescribed
eigenvalue order differ only by simultaneous sign reversal, locally
constant on an overlap. This reversal leaves $\omega$ unchanged. Thus
$\omega$ is a global continuous one-form on $U$, although a global principal
vector field need not exist.

We first note that $\overline U\setminus U\subset\{\tau=0\}$. If a
point of $\overline U$ had $K>0$, a connected coordinate neighborhood
inside $\Omega$ would meet $U$ and therefore belong to the same component.
Since $U$ is connected and $t\ne1$, either $t>1$ everywhere or $0<t<1$
everywhere. In the first case the boundary observation makes $U$ closed
in $M$, hence compact without boundary. Lemma~\ref{lem:weak-gauss} and
\eqref{eq:connection-curvature} would give $\int_U K\dd\mu=0$, a contradiction.

In the second case compactness of $\overline U$ gives $t\leq b<1$.
Otherwise a limit point with $\tau=1$ would belong to $U$. Put $d_0=1-b^2>0$.
Choose a nondecreasing smooth function $\chi:[0,\infty)\to[0,1]$, zero
on $[0,1]$ and one on $[2,\infty)$, with $0\leq\chi'\leq C_\chi$.
For each $\delta>0$, the form $\chi(t/\delta)\omega$ has compact support
in $U$, because $\{\tau\geq\delta\}\cap\overline U$ is compact and
disjoint from $\overline U\setminus U$. Its exterior derivative is continuous by Lemma~\ref{lem:weak-gauss}.
Applying that lemma gives
\begin{align}\label{eq:cutoff-stokes}
 I_\delta:=\int_U\chi(t/\delta)K\dd\mu
 &=\int_U\frac{\chi'(t/\delta)}{\delta}
       \frac{3tp^2-q^2/t}{1-t^2}\dd\mu,\\
 0\leq I_\delta
 &\leq\frac{6C_\chi L^2}{d_0}
     \operatorname{Area}\{x\in U:\delta<t(x)<2\delta\}.\label{eq:layer}
\end{align}
The second line discards the negative square and uses $t/\delta\leq2$
where $\chi'$ is nonzero. The boundary-layer indicator tends pointwise
to zero on the finite-area space $U$, so its area tends to zero.
On the other hand, dominated convergence gives
$I_\delta\to\int_U K\dd\mu>0$. This contradiction proves the proposition.
\end{proof}

\section{Local rigidity for the cubic relation}\label{sec:cubic-local}
\subsection{Spherical rigidity from fourth-order contact}
\label{sec:spherical-contact}

We prove that fourth-order contact with the tangent sphere forces local
coincidence, using a support potential. The curvature relation controls
its Hessian distortion, and planar convexity gives the required
fourth-order growth estimate.

Let $D\subset\R^2$ be a plane domain. A continuous map
$G\in W^{1,2}_{\mathrm{loc}}(D,\R^2)$ is called
\emph{sense-preserving quasiregular} if, for some $\mathcal K\ge1$,
\[
 \|DG\|_{\mathrm{op}}^2\le\mathcal K\det DG
 \quad\text{almost everywhere in }D.
\]
We allow constant maps in this definition and specify nonconstancy when
needed.  In complex notation the distortion inequality is equivalent to
$|G_{\bar z}|\leq k|G_z|$, where $k=(\mathcal K-1)/(\mathcal K+1)<1$.
The following gradient version is contained in the proof of
\cite[Theorem~1.2 and Section~3]{IKO2008}.

\begin{lemma}\label{lem:gradient-convexity}
Let $\Omega\subset\R^2$ be an open convex domain and let
$\psi\in C^2(\Omega,\R)$.  If $\nabla\psi$ is nonconstant and
sense-preserving quasiregular, then $\psi$ is strictly convex or strictly
concave on $\Omega$.
\end{lemma}

\begin{proof}
We give the argument using the critical-point lemma of
Alessandrini--Magnanini \cite[Lemma~3.1]{AM1992}.
Standard planar quasiregular theory gives discreteness and positive
integer local index for a nonconstant map; see the discussion in
\cite[Section~2]{IKO2008} and the background in \cite{Kangasniemi2021}.
Fix $p\in\Omega$ and put
\[
 h_p(x)=\psi(x)-\psi(p)-\nabla\psi(p)\cdot(x-p).
\]
The map $\nabla h_p=\nabla\psi-\nabla\psi(p)$ is still nonconstant
and quasiregular.  Thus $p$ is an isolated critical point of $h_p$,
and the winding number of $\nabla h_p$ around $p$ is positive.
The cited critical-point lemma states that an isolated critical point
of a real $C^1$ function in the plane either is a strict local extremum
with index $1$, or has index $1-L$ for an integer $L\geq1$.
Consequently, $h_p$ is strictly positive or strictly negative in a
sufficiently small punctured neighborhood of $p$.

These signs cannot vary with $p$.  To see this, restrict $\psi$ to any
line through two distinct points of $\Omega$, writing
$g(t)=\psi(p+tv)$ on an open interval.  At every $t$, the tangent-line
remainder
$g(s)-g(t)-g'(t)(s-t)$ has one strict sign for all sufficiently close
$s\ne t$.  If $g'$ had a local maximum at $t$, integration of
$g'(s)-g'(t)\leq0$ would make this remainder nonpositive on the right
and nonnegative on the left, a contradiction.  A local minimum is
excluded in the same way.  A continuous function without local extrema
is strictly monotone, so $g'$ is strictly increasing or strictly
decreasing.  The local signs at the two chosen points therefore agree.
Convexity of $\Omega$ permits every pair of points to be joined in this
way.  All line restrictions are strictly convex, or all are strictly
concave, as required.
\end{proof}

For $\lambda\ne0$, write $\zeta_\lambda=s_\lambda$ for the tangent
sphere graph defined in Section~\ref{sec:common}. Equality of
fourth-order jets means equality of all derivatives through order
four in these tangent-plane graph coordinates.

\begin{lemma}\label{lem:c4-contact}
Let $X$ be a $C^4$ immersion satisfying \eqref{eq:cubic}, and let $p$
be an umbilic with $A(p)=\lambda\Id$, $\lambda\ne0$.
Write the immersion near $p$ as a tangent-plane graph $u$, with
$u(0)=Du(0)=0$.  If
\[
 D^j(u-\zeta_\lambda)(0)=0\qquad(0\leq j\leq4),
\]
then a neighborhood of $p$ is contained in that fixed tangent sphere.
\end{lemma}

\begin{proof}
At the umbilic, \eqref{eq:cubic} gives $c\lambda^2=1$.
After choosing the normal and rescaling, we may assume
$c=\lambda=1$.  Write
\[
 \zeta(x)=1-\sqrt{1-|x|^2},\qquad \eta=u-\zeta.
\]
Taylor's theorem gives $D^j\eta=o(|x|^{4-j})$ for $0\leq j\leq4$.
Both $u$ and $\zeta$ are strictly convex on a fixed small disk.

Use the slope coordinates $y=-Du(x)$.  Since $D^2u(0)=\Id$, they
have a local inverse $x=x(y)$.  Restrict to a full disk $B_R$ in
the $y$-plane, and set
\[
 r=|y|,\qquad w=\sqrt{1+r^2},\qquad
 U(y)=x(y)\cdot y+u(x(y))-1,\qquad \phi=U+w.
\]
Differentiation gives $DU=x$, so $\phi$ is at least $C^2$ and
$\phi(0)=D\phi(0)=0$.  The corresponding inverse and potential for
the sphere are $x_s(y)=-y/w$ and $U_s=-w$.
The points $x(y)$ and $x_s(y)$ minimize, respectively,
$u(x)+y\cdot x$ and $\zeta(x)+y\cdot x$ in the fixed convex disk.
Comparing their minimum values yields
\begin{equation}\label{eq:support-contact-comparison}
 \eta(x(y))\leq\phi(y)\leq\eta(x_s(y)).
\end{equation}
Both inverse maps are $O(r)$ at the origin.  Hence
\begin{equation}\label{eq:support-small-fourth}
 \phi(y)=o(r^4),
\end{equation}
uniformly in the angular variable.

We next express the Hessian in terms of the shape operator.
The unit normal in these coordinates is $n=(y,1)/w$.
Let $P=\Id-n\otimes n$ be orthogonal projection onto its tangent
plane, and define $L:\R^2\to n^\perp$ by $L\xi=P(\xi,0)$.
Then
\[
 dn=\frac{L}{w},\qquad
 L^{\mathsf T}L=\Id-\frac{y\otimes y}{w^2},\qquad
 \spec(L^{\mathsf T}L)=\{1,w^{-2}\}.
\]
The convention $A=-dN$ gives $dX=-w^{-1}A^{-1}L$.
Taking its scalar product with $L\xi$ extracts the horizontal
derivative $Dx=D^2U$.  Since $D^2w=w^{-1}L^{\mathsf T}L$, we obtain
\begin{equation}\label{eq:support-hessian-congruence}
 D^2\phi=w^{-1}L^{\mathsf T}(\Id-A^{-1})L.
\end{equation}

Shrink the disk so that $A$ is positive.  Define $t>0$ by
$\tr A=t+t^3$; the unordered relation \eqref{eq:cubic} then gives
$\spec A=\{t,t^3\}$.  The eigenvalues of $\Id-A^{-1}$ satisfy
\[
 1-t^{-3}=q(t)(1-t^{-1}),\qquad
 q(t)=1+t^{-1}+t^{-2}>1.
\]
Thus $D^2\phi$ is either zero or definite at every point, and in the
latter case its ratio of largest to smallest absolute eigenvalue obeys
\begin{equation}\label{eq:support-distortion}
 \operatorname{cond}(D^2\phi)\leq(1+r^2)q(t).
\end{equation}
The graph formula
$A=(\Id+Du\otimes Du)^{-1}D^2u/\sqrt{1+|Du|^2}$,
expressed in graph coordinates, and the estimates on $D\eta,D^2\eta$
give $A-\Id=o(|x|^2)$.  Since the derivative of $t+t^3$ at $1$
is $4$, it follows that $t-1=o(r^2)$.  Consequently, for a fixed
$C\geq0$ and a smaller disk,
\begin{equation}\label{eq:support-critical-distortion}
 \operatorname{cond}(D^2\phi)\leq3+Cr^2
 \qquad\text{whenever }D^2\phi\ne0.
\end{equation}

The gradient map $G=\nabla\phi$ is sense-preserving quasiregular.
Indeed, $DG=D^2\phi$ is definite or zero, and wherever $D^2\phi\ne0$,
\[
 \|DG\|_{\mathrm{op}}^2
 \leq \operatorname{cond}(D^2\phi)\det DG.
\]
The same inequality is automatic where $D^2\phi=0$, while
\eqref{eq:support-critical-distortion} supplies a fixed distortion bound.
If $G$ is constant, the normalization gives $\phi=0$.
Otherwise Lemma~\ref{lem:gradient-convexity} allows us to choose
$v=\phi$ or $v=-\phi$ strictly convex on $B_R$.
Then $v(0)=Dv(0)=0$, $v>0$ away from $0$, and $D^2v\geq0$.

Let
\[
 m(r)=\frac1{2\pi}\int_0^{2\pi}v(r\cos\theta,r\sin\theta)\dd\theta.
\]
Convexity gives $v_r\geq v/r>0$, hence $m'>0$.
Applied to the radial and tangential unit vectors,
\eqref{eq:support-critical-distortion} gives
\[
 v_{rr}\leq(3+Cr^2)
       \left(\frac{v_r}{r}+\frac{v_{\theta\theta}}{r^2}\right).
\]
This quadratic-form inequality also holds when the Hessian is zero.
Integrating over the full circle removes the angular derivative and
yields
\[
 m''\leq\left(\frac3r+Cr\right)m'.
\]
Fix $0<\rho<R$.  Integration of the logarithmic derivative from
$r$ to $\rho$ gives
\[
 m'(r)\geq m'(\rho)\left(\frac r\rho\right)^3
          \exp\!\left[-\frac C2(\rho^2-r^2)\right]
       \geq c_0r^3,
 \qquad
 c_0=\frac{m'(\rho)}{\rho^3}e^{-C\rho^2/2}>0.
\]
Since $m(0)=0$, another integration gives $m(r)\geq c_0r^4/4$,
contradicting \eqref{eq:support-small-fourth}.
Only the constant-gradient alternative is possible, so $\phi=0$.
Now $U=-w$ and $x=DU=-y/w$ reconstruct
$u=U-y\cdot DU+1=1-w^{-1}$, precisely the sphere graph $\zeta$.
Undoing the normalization proves the lemma.
\end{proof}

\begin{proposition}\label{prop:sphere-propagation}
Let $M$ be connected and without boundary, and let
$X\in C^4(M,\R^3)$ be an immersion satisfying \eqref{eq:cubic}.
If the image of a nonempty open subset of $M$ lies in a fixed round
sphere $\Sigma$, then $X(M)\subset\Sigma$.
If $M$ is compact, then $X(M)=\Sigma$.
\end{proposition}

\begin{proof}
Write $\Sigma=\{z:|z-O|^2=R_0^2\}$, where $R_0>0$, and set
$B=|X-O|^2-R_0^2$ and $E=\operatorname{int}_M B^{-1}(0)$.
The set $E$ is nonempty and open.
If $p\in\overline E$, continuity of the derivatives of $B$ shows
that its entire fourth-order jet vanishes at $p$.
In particular, the tangent plane of $X$ at $p$ is tangent to
$\Sigma$.  Choose graph coordinates at $p$ in which
$O=(0,0,R_0)$, and let $u$ and $\zeta_{1/R_0}$ be the graphs of
the immersion and sphere.  In these coordinates,
\[
 B=(u-\zeta_{1/R_0})(u+\zeta_{1/R_0}-2R_0).
\]
The second factor is nonzero at $p$, so
$u-\zeta_{1/R_0}$ has zero fourth-order jet there.
In particular, $A(p)=R_0^{-1}\Id$ for the chosen local normal.
Lemma~\ref{lem:c4-contact} gives a neighborhood of $p$ in $E$.
Thus $E$ is also closed, and connectedness yields $E=M$.
Finally, if $M$ is compact, its image is closed in $\Sigma$.
The immersion into the two-dimensional sphere is a local
diffeomorphism, so its image is also open; hence it is all of $\Sigma$.
\end{proof}

\subsection{Radial quartic terms and local rotational symmetry}
\label{sec:quartic-rotation}

We work near a nonzero umbilic in the normalization $c=1$.
After choosing the positive normal, write the surface as a graph
$u$ on a full disk $B_\rho\subset\R^2$, with
\[
 u(0)=0,\qquad Du(0)=0,\qquad D^2u(0)=\Id.
\]
Put $r=|x|$ and $\zeta(r)=1-\sqrt{1-r^2}$, so that $\zeta$ is
the tangent sphere of curvature one.  The graph shape operator is
\begin{equation}
 A=(\Id+Du\otimes Du)^{-1}\frac{D^2u}{\sqrt{1+|Du|^2}},
 \qquad
 (\kappa_1-\kappa_2^3)(\kappa_2-\kappa_1^3)=0.
 \label{eq:qr-graph}
\end{equation}
Shrinking the disk makes both principal curvatures positive.
We first determine the cubic and quartic Taylor terms of $u-\zeta$.
When the quartic term is nonzero, an energy estimate forces
$u_\theta=0$, giving rotational symmetry on a smaller full disk.

\begin{lemma}\label{lem:cubic-jet}
If $u\in C^3(B_\rho)$ satisfies \eqref{eq:qr-graph} and the above
normalization, then
\begin{equation}
 D^j(u-\zeta)=o(r^{3-j}),\qquad 0\leq j\leq3.
 \label{eq:qr-cubic-vanishing}
\end{equation}
\end{lemma}

\begin{proof}
Let $P_3$ be the homogeneous cubic Taylor term of $u-\zeta$.
Taylor's formula for the function and its first two derivatives gives
\[
 D^j(u-\zeta-P_3)=o(r^{3-j}),\qquad 0\leq j\leq2.
\]
Since the shape operator of $\zeta$ is exactly $\Id$,
\eqref{eq:qr-graph} gives $A=\Id+D^2P_3+o(r)$.
Writing $T=\tr A$ and $K=\det A$, the curvature relation is
\[
 \mathcal P(A):=K-(T^4-4T^2K+2K^2)+K^3=0.
\]
Its expansion at the identity is
\begin{equation}
 \mathcal P(\Id+E)=16\det E-3(\tr E)^2+O(|E|^3).
 \label{eq:qr-quadratic-part}
\end{equation}
Dividing by $r^2$ along each ray therefore yields
\begin{equation}
 16\det D^2P_3=3(\tr D^2P_3)^2.
 \label{eq:qr-hessian-cone}
\end{equation}

If $P_3$ were nonzero, Lemma~\ref{lem:polynomial} with slope
parameter $3$ would force its degree to be four. This is impossible.
Hence $P_3=0$.  Taylor's formula gives
\eqref{eq:qr-cubic-vanishing} for $j\leq2$, and continuity of
$D^3(u-\zeta)$ gives the assertion for $j=3$.
\end{proof}

\begin{lemma}\label{lem:quartic-jet}
If, in addition, $u\in C^4(B_\rho)$, there is $a\in\R$ such that
\begin{equation}
 u=\zeta+ar^4+R,\qquad D^jR=o(r^{4-j}),\quad 0\leq j\leq4.
 \label{eq:qr-quartic-taylor}
\end{equation}
\end{lemma}

\begin{proof}
By Lemma~\ref{lem:cubic-jet}, the first possible Taylor term of
$u-\zeta$ is a homogeneous quartic $P_4$.
Now $A=\Id+D^2P_4+o(r^2)$, and division of
\eqref{eq:qr-quadratic-part} by $r^4$ gives
$16\det D^2P_4=3(\tr D^2P_4)^2$.
If $P_4=0$, set $a=0$. Otherwise Lemma~\ref{lem:polynomial},
again with slope parameter $3$, gives $P_4=ar^4$ with $a\ne0$.
Taylor's formula, also applied to
the derivatives, gives \eqref{eq:qr-quartic-taylor}.
\end{proof}

The next proposition uses only three derivatives, provided that the
quartic asymptotic expansion holds with the corresponding derivative
estimates.  Lemma~\ref{lem:quartic-jet} supplies these estimates in
the $C^4$ case.

\begin{proposition}\label{prop:quartic-rotation}
Suppose $u\in C^3(B_\rho)$ satisfies \eqref{eq:qr-graph} and
\begin{equation}
 u=\zeta+ar^4+R,\qquad a\ne0,\qquad
 D^jR=o(r^{4-j}),\quad 0\leq j\leq3.
 \label{eq:qr-asymptotic}
\end{equation}
Then, on a smaller full disk, $u$ is radial and, with $h=1+8a$,
\begin{equation}
 u(r)=
 \begin{cases}
  \displaystyle\frac{1-\sqrt{1-hr^2}}{h},&h\ne0,\\[6pt]
  r^2/2,&h=0.
 \end{cases}
 \label{eq:qr-model}
\end{equation}
\end{proposition}

\begin{proof}
For positive principal curvatures, the unordered relation is equivalent
to
\[
 F(p,Q)=\tr A-f(\det A)=0,\qquad
 f(K)=K^{1/4}+K^{3/4},
\]
where $p=Du$, $Q=D^2u$, $g=\Id+p\otimes p$, $W=\sqrt{1+|p|^2}$,
and $A=g^{-1}Q/W$.
Indeed, the sum and product determine the unordered pair
$\{K^{1/4},K^{3/4}\}$.
We linearize this smooth finite-dimensional equation with respect to
rotations of the independent variables.

Set $E=A-\Id$ and
\[
 M=\Id-f'(K)K A^{-1},\qquad K=\det A.
\]
Since $f'(1)=1$ and $f''(1)=-3/8$, expansion gives
\[
 M=E-\frac58(\tr E)\Id+O(|E|^2).
\]
The variation formula $dF=\tr(M\,dA)$ yields
$\alpha:=F_Q=Mg^{-1}/W$ and $\beta:=F_p$.
The matrix $\alpha$ is symmetric, since
$A^{-1}g^{-1}=WQ^{-1}$.
In the Euclidean polar frame, \eqref{eq:qr-asymptotic} gives
\begin{equation}
 \begin{aligned}
 E&=ar^2\diag(12,4)+o(r^2),\\
 \alpha&=2ar^2\diag(1,-3)+o(r^2),\qquad \beta=O(r^3).
 \end{aligned}
 \label{eq:qr-coefficients}
\end{equation}
Each remainder in \eqref{eq:qr-coefficients} retains its stated
order after one application of $r\partial_r$ or $\partial_\theta$.
Equivalently, since $a\ne0$, we may write
\[
 \alpha=2ar^2
 \begin{pmatrix}1+\varepsilon_{11}&\varepsilon_{12}\\
                 \varepsilon_{12}&-3+\varepsilon_{22}\end{pmatrix},
 \qquad
 \varepsilon_{ij},\ r\partial_r\varepsilon_{ij},\
 \ \partial_\theta\varepsilon_{ij}\longrightarrow0
\]
uniformly as $r\to0$.
To see the derivative count, write
\[
 \begin{aligned}
 p-D\zeta&=4ar^2x+DR=O(r^3),\\
 Q-D^2\zeta&=a(8x\otimes x+4r^2\Id)+D^2R.
 \end{aligned}
\]
The hypothesis $D^3R=o(r)$ controls one weighted derivative of
$D^2R=o(r^2)$.
The change in $g^{-1}/W$ from its spherical value is $O(r^4)$,
with the same bound after one weighted derivative.
This proves the assertion for $E$ and then for $\alpha$.
Finally $D_pA=O(r)$, with one weighted derivative also $O(r)$;
contracting with $M=O(r^2)$ gives the stated estimates for $\beta$.
Here $DQ=O(r)$, and the derivatives of the polar frame under
$r\partial_r$ and $\partial_\theta$ are bounded.
Thus all these bounds use at most $D^3u$.

Let $w=u_\theta=-x_2u_{x_1}+x_1u_{x_2}$.
Rotational invariance of $F$ and $u\in C^3$ give the classical equation
\[
 \alpha^{ij}w_{ij}+\beta^jw_j=0,
 \qquad w\in C^2.
\]
Put $s=\log r$.  The polar Hessian identities are
\[
 \begin{aligned}
 w_{rr}&=r^{-2}(w_{ss}-w_s),\\
 \Hess w(e_r,e_\theta)&=r^{-2}(w_{s\theta}-w_\theta),\\
 \Hess w(e_\theta,e_\theta)&=r^{-2}(w_{\theta\theta}+w_s).
 \end{aligned}
\]
Since $a\ne0$, \eqref{eq:qr-coefficients} allows division by
$\alpha_{rr}$ on a smaller punctured disk.  The equation becomes
\begin{equation}
 w_{ss}+2d w_{s\theta}-b w_{\theta\theta}+e w_s+f_0w_\theta=0,
 \label{eq:qr-angular-equation}
\end{equation}
where
\[
 \begin{aligned}
 d=\frac{\alpha_{r\theta}}{\alpha_{rr}},\qquad
 &b=-\frac{\alpha_{\theta\theta}}{\alpha_{rr}},\\
 e=-1+\frac{\alpha_{\theta\theta}}{\alpha_{rr}}
       +\frac{r\beta_r}{\alpha_{rr}},\qquad
 &f_0=-2d+\frac{r\beta_\theta}{\alpha_{rr}}.
 \end{aligned}
\]
The estimates just proved imply, uniformly in $\theta$,
\begin{equation}
 (d,b,e,f_0)\longrightarrow(0,3,-4,0),\qquad
 d_\theta,b_\theta,b_s\longrightarrow0
 \quad(s\longrightarrow-\infty).
 \label{eq:qr-coefficient-limits}
\end{equation}

Set $v=e^{-2s}w$.  Its mean on every circle is zero, because
$w=u_\theta$ and the graph is defined on a full disk.
Equation~\eqref{eq:qr-angular-equation} becomes
\begin{equation}
 v_{ss}+2d v_{s\theta}-b v_{\theta\theta}
 +(e+4)v_s+(4d+f_0)v_\theta+(4+2e)v=0.
 \label{eq:qr-conjugated-equation}
\end{equation}
Consider the energy
\begin{equation}
 \mathcal E(s)=\frac12\int_0^{2\pi}
       \bigl(v_s^2+bv_\theta^2-2v^2\bigr)\dd\theta.
 \label{eq:qr-energy}
\end{equation}
The nonconstant Fourier modes give the periodic Poincar\'e inequality
$\int v^2\leq\int v_\theta^2$.  Consequently this energy is positive definite
when $b$ is sufficiently close to $3$.
Differentiation and periodic integration by parts give the exact identity
\begin{equation}
 \begin{split}
 \mathcal E'={}&\int_0^{2\pi}\Bigl[
   (d_\theta-e-4)v_s^2-(b_\theta+4d+f_0)v_sv_\theta\\
   &\hspace{27mm}-(6+2e)vv_s+\tfrac12b_sv_\theta^2
   \Bigr]\dd\theta.
 \end{split}
 \label{eq:qr-energy-identity}
\end{equation}
These operations require only $v\in C^2$ and $b,d\in C^1$ on
each finite cylinder.

Write $S=\int v_s^2$, $T=\int v_\theta^2$, and $V=\int v^2$.
At the limiting coefficients in \eqref{eq:qr-coefficient-limits},
Poincar\'e gives
\[
 2\mathcal E=S+3T-2V\geq S+T,
 \qquad
 \mathcal E'=2\int_0^{2\pi}vv_s\,d\theta
                  \leq S+V\leq S+T.
\]
The coefficients in the exact identity
\eqref{eq:qr-energy-identity} converge uniformly to these limiting
values. Hence, after shrinking the disk, the same Cauchy and Poincar\'e
estimates give
\[
 \mathcal E'(s)\leq3\mathcal E(s).
\]
The remainder estimates in \eqref{eq:qr-asymptotic} give
$w,w_s,w_\theta=o(r^4)$, hence
$v,v_s,v_\theta=o(e^{2s})$ and $\mathcal E=o(e^{4s})$.
For any fixed $s_0$ in the smaller disk, integration on $[s,s_0]$ gives
\[
 0\leq\mathcal E(s_0)
 \leq e^{3(s_0-s)}\mathcal E(s)\longrightarrow0
 \qquad(s\longrightarrow-\infty).
\]
Thus $v=w=0$, and $u$ is radial.

It remains to integrate the radial curvature equation.
For $q(r)=u'(r)/\sqrt{1+u'(r)^2}$, the meridian and parallel
curvatures are $\kappa_r=q'$ and $\kappa_\theta=q/r$.
The asymptotic expansion gives
\[
 \kappa_r=1+12ar^2+o(r^2),\qquad
 \kappa_\theta=1+4ar^2+o(r^2),
\]
so $\kappa_\theta-\kappa_r^3=-32ar^2+o(r^2)\ne0$ for small
$r>0$.
The unordered relation must therefore take the branch
$q'=q^3/r^3$.
Since $q=r+4ar^3+o(r^3)>0$, integration yields
$q^{-2}=r^{-2}-8a$.
Consequently $u'=r/\sqrt{1-(1+8a)r^2}$, and $u(0)=0$ gives
\eqref{eq:qr-model}.
\end{proof}

\begin{proposition}\label{prop:local-dichotomy}
Under the normalization of this subsection, every $C^4$ graph satisfying
\eqref{eq:qr-graph} agrees on a smaller disk either with $\zeta$ or
with \eqref{eq:qr-model} for some $h\ne1$.
In the second case the graph is a piece of a surface of revolution:
an ellipsoid if $h>0$, a paraboloid if $h=0$, and one sheet of a
two-sheeted hyperboloid if $h<0$.
\end{proposition}

\begin{proof}
Take the coefficient $a$ from Lemma~\ref{lem:quartic-jet}.
If $a=0$, then $u-\zeta=o(r^4)$, and
Lemma~\ref{lem:c4-contact} gives $u=\zeta$ near the origin.
If $a\ne0$, Proposition~\ref{prop:quartic-rotation} applies and
$h=1+8a\ne1$.
The equation $r^2=2u-hu^2$ identifies the three indicated quadrics.
\end{proof}

\section{Global completion and classification}\label{sec:cubic-global}
\subsection{Propagation on a principal cylinder}
\label{sec:principal-propagation}

In homogeneous three-manifolds with four-dimensional isometry group,
G\'alvez and Mira obtain rotational symmetry for elliptic Weingarten
spheres under a bounded-curvature condition on the canonical rotational
model \cite{GM2021}. For the cubic relation in $\R^3$, we have already
constructed a rotational neighborhood. Its continuation uses the
principal-coordinate equation below.

We extend the local rotational annulus by matching curvature data on
complete principal circles. For a $C^3$ immersion, the curvature function
$t$ is only $C^1$. We therefore use the Gauss equation in divergence form
and prove the energy identity at that regularity. Throughout this subsection
the curvature relation is normalized to $c=1$, and the regions under
consideration are positive definite and contain no umbilics.

\begin{lemma}\label{lem:principal-coordinates}
On a local principal-frame patch, the principal coframe determines closed $C^1$ one-forms
\begin{equation}\label{eq:closed-principal-forms}
 \eta_1=\frac{t}{\sqrt{|D|}}\theta_1,
 \qquad \eta_2=|D|^{3/2}\theta_2.
\end{equation}
Their local primitives are $C^2$ coordinates. If $\gamma$ is an embedded
closed principal curve tangent to the $t$-principal line, a sufficiently
narrow annulus about $\gamma$ has coordinates
$(x\bmod\ell,y)$, where $\ell>0$, $\gamma=\{y=0\}$, and
$dx=\eta_1$, $dy=\eta_2$. In these coordinates,
\begin{equation}\label{eq:principal-metric}
 g=\frac{|D|}{t^2}\,dx^2+|D|^{-3}\,dy^2.
\end{equation}
\end{lemma}

\begin{proof}
Put $p=e_1t$, $q=e_2t$, $\alpha=\omega(e_1)$, and
$\beta=\omega(e_2)$. The two Codazzi equations give
\[
 q=\alpha(t-t^3),\qquad 3t^2p=\beta(t-t^3),
 \qquad
 \alpha=\frac{q}{tD},\quad \beta=\frac{3tp}{D}.
\]
The bracket relation $[e_1,e_2]=-\alpha e_1-\beta e_2$ yields
$d\theta_1=\alpha\theta_1\wedge\theta_2$ and
$d\theta_2=\beta\theta_1\wedge\theta_2$. For
$f(t)=t/\sqrt{|D|}$ and $k(t)=|D|^{3/2}$, one has
\[
 \frac{f'(t)}{f(t)}=\frac1{tD},\qquad
 \frac{k'(t)}{k(t)}=-\frac{3t}{D}.
\]
These identities prove the closedness in
\eqref{eq:closed-principal-forms}. A primitive of a $C^1$ closed one-form
is $C^2$. Since $\eta_1\wedge\eta_2$ never vanishes, the inverse function
theorem gives local coordinates with $C^2$ inverse.

Orient $e_1$ along $\gamma$ and extend this choice to a narrow annulus;
the orientation of the surface determines $e_2$. The period of $\eta_2$
on $\gamma$ is zero, whereas
\[
 \ell=\int_\gamma\eta_1
     =\int_\gamma\frac{t}{\sqrt{|D|}}\,ds_g>0.
\]
Closedness therefore gives a single-valued primitive $y$ of $\eta_2$
and a primitive $x\bmod\ell$ of $\eta_1$ on the annulus. Along $\gamma$,
$x$ increases strictly and traverses $\R/\ell\mathbb Z$ once.
Compactness of $\gamma$ and the local inverse function theorem give a
uniform smaller annulus on which these are cylinder coordinates.
Equivalently, the flow of $Y=e_2/|D|^{3/2}$ extends the initial $x$
coordinate while increasing $y$ at unit speed, because
$\eta_1(Y)=0$, $\eta_2(Y)=1$, and both forms are closed.
Although this flow is only $C^1$, its inverse coordinates satisfy
$dx=\eta_1$, $dy=\eta_2$ and are thus $C^2$.
Finally, $\theta_1=(\sqrt{|D|}/t)dx$ and
$\theta_2=|D|^{-3/2}dy$, which proves \eqref{eq:principal-metric}.
\end{proof}

Write $m=\sqrt{|D|}/t$ and $n=|D|^{-3/2}$. In principal coordinates
the connection is the continuous one-form
\[
 \omega=-\frac{m_y}{n}\,dx+\frac{n_x}{m}\,dy.
\]
Lemma~\ref{lem:weak-gauss} gives, in distributions,
$\partial_x(n_x/m)+\partial_y(m_y/n)=-Kmn$.
This identity remains valid in the $C^2$ coordinates above: locally the
connection is a $C^1$ reference connection plus the differential of a
$C^1$ angle, and this decomposition retains its regularity under a
$C^2$ coordinate change. With $\varepsilon=\operatorname{sign}D$, the identities
\[
 mn=\frac1{t|D|},\qquad
 \frac{n'(t)}{m}=\frac{3\varepsilon t^2}{|D|^3},\qquad
 \frac{m'(t)}{n}=-\frac{\varepsilon|D|}{t^2}
\]
give the weak curvature equation
\begin{equation}\label{eq:principal-divergence}
 \varepsilon\left\{\partial_y\bigl(B(t)t_y\bigr)
                  -\partial_x\bigl(C(t)t_x\bigr)\right\}
       =\frac{t^3}{|D|},\qquad
 B(t)=\frac{|D|}{t^2},\quad C(t)=\frac{3t^2}{|D|^3}.
\end{equation}
Both $B$ and $C$ are positive on either interval $(0,1)$ or $(1,\infty)$.

\begin{lemma}\label{lem:weak-energy}
Let $J$ be an open interval and $S_\ell=\R/\ell\mathbb Z$, with $\ell>0$.
Suppose $w,\mathcal A\in C^1(S_\ell\times J)$,
$\mathcal B\in C^0(S_\ell\times J)$, $\mathcal A>0$, and
\begin{equation}\label{eq:linear-weak-wave}
 w_{yy}-\partial_x(\mathcal A w_x)=\mathcal B w
\end{equation}
in distributions. If $w=w_y=0$ on $S_\ell\times\{y_0\}$ for some
$y_0\in J$, then $w=0$ on $S_\ell\times J$.
\end{lemma}

\begin{proof}
Set $p=w_y$ and $q=w_x$. These functions are continuous and satisfy
$p_y=\partial_x(\mathcal A q)+\mathcal B w$ and $q_y=\partial_xp$
in distributions. Fix a compact time interval inside $J$ containing $y_0$.
Let a subscript
$\delta$ denote convolution in the periodic variable $x$ alone with a
nonnegative smooth approximate identity $\rho_\delta$, supported in
$(-\delta,\delta)$ before periodization. Then
\[
 (p_\delta)_y=\partial_x(\mathcal A q)_\delta+(\mathcal B w)_\delta,
 \qquad(q_\delta)_y=\partial_xp_\delta,
 \qquad(w_\delta)_y=p_\delta.
\]
The right-hand sides are continuous. A continuous function whose
distributional $y$ derivative is continuous equals the integral of that
derivative, up to a function of $x$; thus these are classical $y$
derivatives. This justifies differentiating the regularized energy.

The only commutator needed is
$r_\delta=\partial_x(\mathcal A q)_\delta
          -\partial_x(\mathcal A q_\delta)$. Without differentiating $q$,
\begin{align*}
 r_\delta(x,y)
 &=\int\rho_\delta'(z)
       [\mathcal A(x-z,y)-\mathcal A(x,y)]q(x-z,y)\,dz
       -\mathcal A_x(x,y)q_\delta(x,y).
\end{align*}
For a continuous function $f$, let $\omega_f(\delta)$ denote its uniform
modulus of continuity in $x$ on the chosen compact cylinder. Subtracting
$q(x,y)$ inside the integral and using
\[
 \int\rho_\delta'(z)[\mathcal A(x-z,y)-\mathcal A(x,y)]\,dz
       =(\mathcal A_x)_\delta(x,y)
\]
gives
\begin{equation}\label{eq:spatial-commutator}
 \|r_\delta\|_\infty
 \leq\|q\|_\infty\omega_{\mathcal A_x}(\delta)
 +(M_1+1)\|\mathcal A_x\|_\infty\omega_q(\delta)
 \longrightarrow0,
\end{equation}
where $M_1=\int|z\rho_\delta'(z)|\,dz$ is independent of $\delta$.

Define
\[
 E_\delta(y)=\frac12\int_{S_\ell}
       (p_\delta^2+\mathcal A q_\delta^2+w_\delta^2)\,dx.
\]
The regularized first equation is
$(p_\delta)_y=\partial_x(\mathcal A q_\delta)
 +(\mathcal B w)_\delta+r_\delta$. Periodic integration by parts therefore gives
\[
 E_\delta'(y)=\int_{S_\ell}
 \left\{\frac12\mathcal A_yq_\delta^2
       +p_\delta[(\mathcal B w)_\delta+r_\delta+w_\delta]\right\}\,dx.
\]
The principal terms cancel exactly. All functions in this formula converge
uniformly on the compact cylinder, including the commutator by
\eqref{eq:spatial-commutator}. Hence the limiting energy is $C^1$ and satisfies
\begin{equation}\label{eq:weak-energy-identity}
 \begin{split}
 E(y)&=\frac12\int_{S_\ell}(w_y^2+\mathcal A w_x^2+w^2)\,dx,\\
 E'(y)&=\int_{S_\ell}
       \left\{\frac12\mathcal A_yw_x^2+(\mathcal B+1)ww_y\right\}\,dx.
 \end{split}
\end{equation}
The positive lower bound for $\mathcal A$ on the compact cylinder gives
$|E'|\leq C_0E$. The initial conditions also give $w_x(\cdot,y_0)=0$,
so $E(y_0)=0$. The differential inequality in both time directions gives
$E=0$. Any point of $J$ lies with $y_0$ in a compact subinterval of $J$,
which proves the assertion.
\end{proof}

\begin{proposition}\label{prop:periodic-uniqueness}
Let $t\in C^1(S_\ell\times J)$ satisfy
\eqref{eq:principal-divergence}, with values in one of $(0,1)$ or $(1,\infty)$.
Let $T(y)$ be a smooth rotational solution of the same equation with values
in the same interval. If $t=T$ and $t_y=T_y$ on a complete circle
$y=y_0$, then $t(x,y)=T(y)$ throughout the common cylinder.
If $t$ and $T$ are the curvature functions of immersed cylinders in these
principal coordinates, the immersions have identical first and second
fundamental forms and differ by a fixed Euclidean rigid motion. If they
already agree on an open subcylinder, this motion is fixed by that agreement.
\end{proposition}

\begin{proof}
Choose a primitive $J_0$ of $B$ on the relevant interval. Since $B>0$,
$J_0$ is a smooth one-to-one change of variable. Put $U=J_0(t)$ and
define functions of $U$ by
\[
 a_0(U)=\frac{C(J_0^{-1}(U))}{B(J_0^{-1}(U))},\qquad
 G(U)=\frac{\varepsilon(J_0^{-1}(U))^3}
             {|1-(J_0^{-1}(U))^2|}.
\]
The classical first derivative identities $U_y=B(t)t_y$ and
$U_x=B(t)t_x$ transform \eqref{eq:principal-divergence} into
$U_{yy}-\partial_x(a_0(U)U_x)=G(U)$ in distributions.
For $V(y)=J_0(T(y))$ and $w=U-V$, the identity $V_x=0$ gives
\[
 w_{yy}-\partial_x(a_0(U)w_x)=b_0(x,y)w,
 \qquad b_0=\int_0^1G'(V+\lambda(U-V))\,d\lambda.
\]
On every compact common cylinder, $a_0(U)$ is positive and $C^1$, and
$b_0$ is continuous and bounded. The segment between $U$ and $V$ stays
in the interval $J_0((0,1))$ or $J_0((1,\infty))$, as appropriate.
The initial data give $w=w_y=0$, so Lemma~\ref{lem:weak-energy} yields
$U=V$ and therefore $t=T$.

Formula~\eqref{eq:principal-metric} and the diagonal principal curvatures
now give identical first and second fundamental forms. In the $C^2$
principal coordinates the immersions are at least $C^2$ and their adapted
orthonormal frames are $C^1$. Along any piecewise $C^1$ path, the
Gauss--Weingarten equations are linear ordinary differential equations
with continuous coefficients. If the position and adapted frame agree at
one point, uniqueness makes the immersions agree along the path. Since the
cylinder is connected, one rigid motion identifies them everywhere.
\end{proof}

\subsection{Completion of a rotational annulus}
\label{sec:compact-completion}

A rotational annulus supplies the complete circles required by the
preceding uniqueness result. We follow their meridian geodesics on the
compact surface and compare the resulting family with a fixed rotational
model. At a finite endpoint away from a pole, continuity of the original
shape operator provides the data for a new principal cylinder, so the
matching continues.

\begin{proposition}\label{prop:annulus-completion}
Let $M$ be a connected compact smooth surface without boundary and
$X:M\to\R^3$ a $C^3$ immersion satisfying \eqref{eq:cubic} with $c=1$.
Suppose an open annulus $U\subset M$ is mapped diffeomorphically onto a
rotational annulus consisting of complete latitude circles, each traversed
once. Assume that $A$ is positive definite and nonumbilic on $U$, with
latitude curvature $t$ and meridian curvature $t^3$.
Then $X$ is an embedding onto a rotational ellipsoid. After a rigid motion,
its image has the equation
\begin{equation}\label{eq:completed-ellipsoid}
 h r^2+h^2(z-z_0)^2=1
\end{equation}
for some $h>0$, $h\ne1$, where $r$ is the distance to the rotation axis.
\end{proposition}

\begin{proof}
Parametrize the meridian of the given annulus by arc length $s$ and write
\[
 Q(s,\theta)=(r(s)\cos\theta,r(s)\sin\theta,z(s)),
 \qquad \theta\in\R/2\pi\mathbb Z.
\]
Choose the direction of $s$ and the axial coordinate so that, for the
meridian angle $\psi$, the positive principal curvatures give
\begin{equation}\label{eq:meridian-system}
 r'=\cos\psi,\qquad z'=\sin\psi=tr>0,
 \qquad \psi'=t^3.
\end{equation}
Primes here denote derivatives with respect to $s$. Differentiating
$t=\sin\psi/r$ gives $t'=(t^3-t)r'/r$. Consequently,
$[(t^{-2}-1)/r^2]'=0$. Denote this constant by $1-h$. Then
\begin{equation}\label{eq:meridian-first-integral}
 t=\bigl(1+(1-h)r^2\bigr)^{-1/2},\qquad
 (r')^2=\frac{1-hr^2}{1+(1-h)r^2}.
\end{equation}
The case $h=1$ would give $t=1$, contrary to the initial annulus.
The derivation did not divide by $r'$, so it also applies at an equator.
Since $D=1-t^2$, the first integral also gives
\begin{equation}\label{eq:model-period}
 D=(1-h)t^2r^2,
 \qquad
 \ell=\int_\gamma\frac{t}{\sqrt{|D|}}\,ds_g
      =\frac{2\pi tr}{\sqrt{|D|}}
      =\frac{2\pi}{\sqrt{|1-h|}}
\end{equation}
on every model latitude $\gamma$. Thus its principal-coordinate period
is independent of the latitude.

For $h>0$ the complete model is \eqref{eq:completed-ellipsoid}. The
parametrization
\[
 r=\frac{\sin v}{\sqrt h},\qquad
 z=z_0-\frac{\cos v}{h},\qquad 0<v<\pi,
\]
has a positive arc-length element, is regular at the equator, and has
$t>0$, $t\ne1$ except at the two poles. Its maximal meridian interval
without poles is finite at both ends. If $h\leq0$, one meridian end is
a pole and the other has $r\to\infty$, with
\begin{equation}\label{eq:infinite-meridian-end}
 \left|\frac{ds}{dr}\right|
   =\sqrt{\frac{1+(1-h)r^2}{1-hr^2}}\geq1.
\end{equation}
Thus this end has infinite arc length. Let $I$ be the maximal meridian
interval without poles containing the initial annulus in either case.
Every finite interior point of $I$ has $r>0$, $t>0$, and $t\ne1$.
The inequality $z'>0$ makes $Q$ injective on
$(\R/2\pi\mathbb Z)\times I$.

Fix an initial latitude $\gamma(\theta)\in U$ corresponding to
$Q(s_0,\theta)$ and let $V(\theta)$ be its forward unit meridian vector.
The local inverse of the $C^3$ immersion gives $\gamma\in C^3$ and
$V\in C^2$. Meridians in the initial rotational annulus are geodesics
for its metric $ds^2+r(s)^2d\theta^2$. Let $\operatorname{Exp}$ denote
the exponential map of the induced metric, and define
\begin{equation}\label{eq:global-geodesic-family}
 F(\theta,s)=\operatorname{Exp}_{\gamma(\theta)}
                  ((s-s_0)V(\theta)).
\end{equation}
This family is defined for every $s\in\R$. Indeed, $g\in C^2$ has
$C^1$ Christoffel symbols, so its geodesic vector field on the tangent
bundle is $C^1$. The vector field preserves unit speed. Compactness of
the unit tangent bundle prevents escape in finite time and gives complete
geodesics by local ordinary differential equation continuation.
Dependence on the initial data is $C^1$, hence $F$ is $C^1$ jointly in
$(\theta,s)$. The velocity $F_s$ is the velocity component of this flow.
On an initial open interval about $s_0$, one has $X\circ F=Q$.

Let $J\subset I$ be the maximal open interval containing $s_0$ on which
this equality holds for every $\theta$. Suppose that $b\in I$ is a
finite right endpoint of $J$. Continuity of $F$ and its first derivatives
gives $X\circ F=Q$ and equality of first derivatives at $s=b$, so
\begin{equation}\label{eq:endpoint-first-order}
 |F_s|_g=1,\qquad |F_\theta|_g=r(b)>0,
 \qquad g(F_s,F_\theta)=0.
\end{equation}
The curvature at this circle is determined by continuity on $M$.
On the matched open side, $dQ$ has rank two, and therefore so does $dF$.
Locally choose a linear projection $\pi:\R^3\to\R^2$ for which
$\pi\circ X$ is a $C^3$ coordinate map. On that side,
$F=(\pi\circ X)^{-1}\circ\pi\circ Q$ is $C^3$, so the original shape
operator agrees with the model in the matched principal directions.
In particular $K(F(\theta,s))=t(s)^4$ there. Continuity of $K$ first gives
$K(F(\theta,b))=t(b)^4>0$. The endpoint circle thus has a positive
curvature neighborhood with its unique positive normal. In that
neighborhood the shape operator is $C^1$, and the principal relations
on the old side pass to the endpoint:
\begin{equation}\label{eq:endpoint-shape-continuity}
 A F_\theta=t(b)F_\theta,
 \qquad A F_s=t(b)^3F_s\qquad\text{at }s=b.
\end{equation}
This argument requires only the first derivatives of $F$ at the endpoint.
The endpoint remains in the positive curvature component containing $U$,
since each matched meridian segment joins it to $U$ through positive
curvature points.

The endpoint circle is embedded in $M$: equality of two source points
would imply equality of the corresponding points of the model latitude,
which is traversed once. Equations~\eqref{eq:endpoint-first-order}
and the $C^1$ inverse function theorem give local invertibility of $F$
near this circle. They give an embedded annulus of uniform width.
To see injectivity after shrinking, a sequence of distinct coincident
pairs in thinner annuli would have limiting points on the endpoint
circle. Injectivity of that circle forces the same limiting angle, and
local invertibility then excludes the pairs. The annulus can also be
chosen with $t$ bounded away from $0$ and $1$.

We construct principal coordinates on this annulus before applying
Proposition~\ref{prop:periodic-uniqueness}. The $t^3$ spectral line is
$C^1$. Choose its local unit sections with positive scalar product with
$F_s$; near the endpoint circle this choice is possible by
\eqref{eq:endpoint-shape-continuity}. The signs agree on overlaps,
so they define a $C^1$ field $e_2$ equal to $F_s$ on the circle.
Choose $e_1$ so that $(e_1,e_2)$ is positively oriented; it is tangent
to the latitude. The circle is an
$e_1$ curve. Lemma~\ref{lem:principal-coordinates} gives $C^2$ coordinates
$(x\bmod\ell,y)$ with the endpoint circle at $y=0$, and
\[
 \partial_y=Y=\frac{e_2}{|D|^{3/2}}.
\]
The model has the same period $\ell$ and the same initial circle.
Choose its principal coordinates to agree with those on the matched side.
There the curvature is $T(y)$, independent of $x$. Since the original
$t$ is $C^1$ and $Y$ is continuous, both $t$ and $Yt=t_y$ extend
continuously to the endpoint circle. Their values agree there with
$T$ and $T_y$ on the entire circle.
Indeed, on the matched side $e_2=F_s$, and hence
\begin{equation}\label{eq:model-transverse-data}
 \frac{dy}{ds}=|D|^{3/2},
 \qquad t_y=\frac{t'}{|D|^{3/2}}.
\end{equation}
Equations~\eqref{eq:model-period} and \eqref{eq:model-transverse-data}
show that the source and model cylinders have the same period and the
same Cauchy data $t,t_y$ at the endpoint.

Proposition~\ref{prop:periodic-uniqueness} now matches the new cylinder
with the same rotational model, with its rigid motion fixed by the old
side. Its unit meridians are geodesics and have, at $s=b$, the same
positions and velocities as $F$. Uniqueness for the $C^1$ geodesic vector
field identifies them with $F$. Thus $J$ extends beyond $b$, a
contradiction. The same argument applies at a left endpoint, and $J=I$.

If $h\leq0$, every finite point of the infinite model end is therefore
contained in $X(M)$. Equation~\eqref{eq:infinite-meridian-end} implies
that $r$ is arbitrarily large there, contradicting compactness of $X(M)$.
Hence $h>0$, and the matching covers the ellipsoid with its two poles
removed. Write $I=(s_-,s_+)$ for its finite meridian interval.

The family $F$ was already defined at $s_-$ and $s_+$. Since
$|F_\theta|_g=r(s)\to0$, its $C^1$ regularity gives
$F_\theta(\theta,s_\pm)=0$. Each endpoint circle is consequently one
source point $p_\pm$. Their images are the distinct poles, so
$p_-\ne p_+$. At the lower pole the vectors
$dX_{p_-}F_s(\theta,s_-)=Q_s(s_-,\theta)$ traverse the unit circle in
the tangent plane exactly once. At the upper pole use the inward
velocities $-F_s(\theta,s_+)$. Since $dX_p$ is a linear isometry onto
the tangent plane, the source velocities also traverse each complete
unit tangent circle once.

The exponential map at either pole is $C^1$. Linearizing the geodesic
system at zero initial velocity gives $D\operatorname{Exp}_p(0)=\Id$,
so the inverse function theorem gives a full normal coordinate disk.
Geodesic uniqueness identifies the matched family near that pole with
$\operatorname{Exp}_p(u v(\theta))$, where $u\geq0$ and the velocities
$v(\theta)$ traverse the unit circle once. Thus the matched set, with
the two poles added, is open in $M$. It is also the continuous image of
the compact cylinder $(\R/2\pi\mathbb Z)\times[s_-,s_+]$, hence is
closed in $M$. Connectedness implies that it is all of $M$.
Injectivity of the model cylinder and the distinct pole images now give
injectivity of $X$ on $M$. A compact injective immersion is an embedding,
which proves the proposition.
\end{proof}

For the local graph $u=\zeta+a r^4+o(r^4)$ produced in
Proposition~\ref{prop:quartic-rotation}, the model parameter is
$h=1+8a$. Indeed, the lower pole of \eqref{eq:completed-ellipsoid} is
the graph
\[
 u(r)=\frac{1-\sqrt{1-hr^2}}{h}
     =\frac{r^2}{2}+\frac{h r^4}{8}+O(r^6),
\]
whereas $\zeta(r)=r^2/2+r^4/8+O(r^6)$. Compact completion therefore
forces $a>-1/8$.

\subsection{Proof of the cubic classification}\label{sec:classification}
\begin{proof}[Proof of Theorem~\ref{thm:main}]
A connected component of a manifold is open and closed, so each component
of the compact surface $M$ is again compact without boundary. Work on
one component. Lemma~\ref{lem:positive} gives $c>0$. Rescale to $c=1$
as in Section~\ref{sec:existence}. Proposition~\ref{prop:umbilic}
supplies a nonzero umbilic. With its positive normal, represent a
neighborhood as a graph $u$ satisfying
$u(0)=Du(0)=0$ and $D^2u(0)=\Id$.

By Proposition~\ref{prop:local-dichotomy}, this graph is either a piece
of the unit sphere or a nonspherical rotational quadric. In the spherical
case Proposition~\ref{prop:sphere-propagation} gives the entire sphere.
The immersion into that sphere is a proper local diffeomorphism and hence
a covering map. Since the sphere is simply connected, the covering has
one sheet, so $X$ is an embedding on this component.
In the other case write its parameter as $h=1+8a\ne1$.
A sufficiently small annulus around the graph center is a complete,
single-covered rotational annulus. Its principal curvatures are positive
and distinct, with parallel curvature $t$ and meridian curvature $t^3$.
Proposition~\ref{prop:annulus-completion} therefore embeds the component
as an ellipsoid of revolution.

We now return to the original immersion and constant $c$ before normalization.
For the semiaxes and the converse, parametrize an ellipsoid by
\[
 Q(v,\theta)=(a_e\sin v\cos\theta,a_e\sin v\sin\theta,b_e\cos v),
 \qquad 0<v<\pi,
\]
where $a_e,b_e>0$, and put
$E=a_e^2\cos^2v+b_e^2\sin^2v$.
For the inward normal its parallel and meridian curvatures are
\begin{equation}\label{eq:axis-relation}
 \kappa_\theta=\frac{b_e}{a_e\sqrt E},\qquad
 \kappa_v=\frac{a_e b_e}{E^{3/2}},\qquad
 \frac{\kappa_v}{\kappa_\theta^3}=\frac{a_e^4}{b_e^2}.
\end{equation}
The pole curvatures, obtained by continuity, both equal $b_e/a_e^2$.
Consequently the unordered relation at a pole forces
$c=a_e^4/b_e^2$, regardless of the labels elsewhere.
Equivalently $b_e=a_e^2/\sqrt c$. Conversely,
\eqref{eq:axis-relation} verifies the same fixed relation on the
whole ellipsoid, including the poles by continuity. A sphere has
radius $\sqrt c$.
\end{proof}

The nonzero-umbilic argument and the completion of a specified rotational
annulus require only $C^3$ regularity. Proposition~\ref{prop:quartic-rotation}
also applies to a $C^3$ graph with the stated quartic remainder estimates.
For a general $C^3$ graph, Lemma~\ref{lem:cubic-jet} removes the cubic
spherical difference but does not supply those fourth-order estimates.
This proof therefore neither classifies arbitrary $C^3$ immersions nor
determines the optimal regularity.

\section*{Acknowledgments}
The author thanks Professor Bobo Hua for his support.

\section*{Disclosure on AI assistance}
The author used AI-assisted tools, principally ChatGPT. The author verified
and completed all mathematical arguments, and takes full responsibility for
the content of the paper.

\begingroup
\raggedright
\bibliographystyle{amsplain}
\bibliography{references}
\endgroup
\end{document}